\documentclass[times,review,10pt]{elsarticle}

\usepackage{amssymb}
\usepackage{amsmath}
\usepackage{amsthm}
\usepackage{algorithm, algorithmic}
\usepackage{array}
\usepackage{adjustbox}
\usepackage{subfigure}
\usepackage{hyperref}
\usepackage{subcaption}
\usepackage{rotating}
\usepackage{color,xcolor}
\usepackage{makecell}
\usepackage{booktabs} % 用于更好的表格线条控制，但本例未使用边框，仅用于可能的三线表需求
\usepackage{array}   % 增强列格式
\newtheorem{theorem}{Theorem}[section]
\newtheorem{lemma}{Lemma}[section]
\newtheorem{corollary}{Corollary}[section]
\newtheorem{proposition}{Proposition}[section]
\newtheorem{definition}{Definition}[section]
\newtheorem{remark}{Remark}[section]

\newcommand{\thickvline}{\vrule width 0.6pt}
\newcommand{\thickhline}{\noalign{\hrule height 0.8pt}} % 1.5pt粗线
\newcommand{\mediumhline}{\noalign{\hrule height 0.6pt}}  % 1pt中线

\journal{Nuclear Physics B}

\begin{document}

\begin{frontmatter}

%% Title, authors and addresses

%% use the tnoteref command within \title for footnotes;
%% use the tnotetext command for theassociated footnote;
%% use the fnref command within \author or \affiliation for footnotes;
%% use the fntext command for theassociated footnote;
%% use the corref command within \author for corresponding author footnotes;
%% use the cortext command for theassociated footnote;
%% use the ead command for the email address,
%% and the form \ead[url] for the home page:
%% \title{Title\tnoteref{label1}}
%% \tnotetext[label1]{}
%% \author{Name\corref{cor1}\fnref{label2}}
%% \ead{email address}
%% \ead[url]{home page}
%% \fntext[label2]{}
%% \cortext[cor1]{}
%% \affiliation{organization={},
%%             addressline={},
%%             city={},
%%             postcode={},
%%             state={},
%%             country={}}
%% \fntext[label3]{}

\title{Effective algorithms for tensor-train decomposition via the UTV framework
}

%% use optional labels to link authors explicitly to addresses:
%% \author[label1,label2]{}
%% \affiliation[label1]{organization={},
%%             addressline={},
%%             city={},
%%             postcode={},
%%             state={},
%%             country={}}
%%
%% \affiliation[label2]{organization={},
%%             addressline={},
%%             city={},
%%             postcode={},
%%             state={},
%%             country={}}

\author[author1]{Yuchao Wang} %% Author name
%% Author affiliation
\ead{ycwang@cdut.edu.cn}
\affiliation[author1]{organization={School of Mathematical Sciences},%Department and Organization
            addressline={Chengdu University of Technology}, 
            city={Chengdu},
            postcode={610059}, 
            state={Sichuan},
            country={China}}

\author[author2]{Maolin Che} %% Author name
\ead{chncml@outlook.com}
%% Author affiliation
\affiliation[author2]{organization={School of Mathematics and Statistics},%Department and Organization
	addressline={Guizhou University}, 
	city={Guiyang},
	postcode={550025}, 
	state={Guizhou},
	country={China}}

\author[author3]{Yimin Wei \corref{cor1}} %% Author name
 %\thanks{Corresponding author. E-mail address: ymwei@fudan.edu.cn
%% Author affiliation
\ead{ymwei@fudan.edu.cn}
\cortext[cor1]{Corresponding author.}
\affiliation[author3]{organization={School of Mathematical Sciences and Key Laboratory of Mathematics for Nonlinear Sciences},%Department and Organization
	addressline={Fudan	University}, 
	city={Shanghai},
	postcode={200433}, 
	state={Shanghai},
	country={China}}
	
%% Abstract
\begin{abstract}
%% Text of abstract
The tensor-train (TT) decomposition is widely used to compress large tensors into a more compact form by exploiting their inherent data structures. A fundamental approach for constructing the TT format is the well-known TT-SVD method, which performs singular value decompositions (SVDs) on the successive unfolding matrices sequentially. But in practical applications, it is often unnecessary to compute full SVDs. In this article, we propose the TT-UTV method, a general framework that incorporates the family of UTV-type algorithms to compute the TT decomposition. As an extension of TT-SVD, it can choose rank-revealing UTV factorizations for the extraction of core tensors, enabling more efficient handling of large-scale intermediate unfolding matrices. We further analyze the error propagation of each step within the TT-UTV framework, and derive a tight overall error bound. Based on this analysis, we recommend distinct sweep patterns for the ULV and URV cases, and formulate the rank-adaptive algorithms with prescribed accuracy.
%By leveraging the advantages of rank-revealing UTV factorizations, the method efficiently handles large-scale tensors at a reduced computational cost. We analyze the error bounds on the accuracy of these algorithms in both the URV and ULV cases, and then recommend different sweep patterns for these two cases. Based on the theoretical analysis, we also formulate the rank-adaptive algorithms with prescribed accuracy. 
Numerical experiments on various datasets, including magnetic resonance imaging data completion, demonstrate their good performance in practice.
\end{abstract}

%%Graphical abstract
%\begin{graphicalabstract}
%\includegraphics{grabs}
%\end{graphicalabstract}

%%Research highlights
\iffalse
\begin{highlights}
\item We propose the TT-UTV method for efficient tensor-train decomposition, which employs rank-revealing UTV factorizations to circumvent computationally expensive full SVDs on successive large-scale matrices. This general framework accommodates the entire family of UTV-type algorithms and provides a  scaled computational advantage over conventional TT-SVD, building upon their inherent efficiency at the single-matrix level.
\item We establish a rigorous error analysis for both ULV and URV cases, theoretically revealing how different sweep patterns affect the propagation of local truncation errors. Based on this analysis, we recommend a left-to-right sweep for TT-ULV and a right-to-left sweep for TT-URV, respectively.
\item Experiments are conducted on multiple tensor datasets, including the magnetic resonance imaging data completion. The results demonstrate that the proposed TT-UTV algorithms achieve comparable accuracy to the standard TT-SVD method, while attaining higher computational efficiency. 
\end{highlights}
\fi

%% Keywords
\begin{keyword}
	Tensor-train decomposition\sep UTV decomposition \sep Rank-adaptive algorithms \sep Data completion.
%% keywords here, in the form: keyword \sep keyword

%% PACS codes here, in the form: \PACS code \sep code

%% MSC codes here, in the form: \MSC code \sep code
%% or \MSC[2008] code \sep code (2000 is the default)

\end{keyword}

\end{frontmatter}

%% Add \usepackage{lineno} before \begin{document} and uncomment 
%% following line to enable line numbers
%% \linenumbers

%% main text
%%

\section{Introduction}

Tensors provide a natural representation for high-dimensional data and nonlinear interactions by preserving their inherent multi-way structure. This makes them particularly suitable for modeling complex real-world systems, such as color images and videos \cite{Salman2023fast, lu2011survey}, hypergraphs \cite{cui2024discrete}, genetic sequences \cite{hore2016tensor}, and social interactions \cite{ding2018tensor, Che2024tensor}. By maintaining the structural integrity of multi-dimensional data, tensor methods facilitate more systematic and interpretable analysis across diverse fields including machine learning, computer vision, bioinformatics, network science, and beyond.
%Real-world high-dimensional data are usually represented as tensors, such as video data, hypergraphs, genetic sequences, and social network interactions \cite{lu2011survey,hore2016tensor,cui2024discrete,  ding2018tensor}. Tensors provide a natural way to model multi-dimensional data and multi-modal nonlinear interactions.
However, directly processing these higher-order tensors demands a significant amount of resources, as the time and space complexity grow exponentially with tensor order. This phenomenon is known as the ``curse of dimensionality'' \cite{Indyk1998}. Fortunately, tensors in the real world often exhibit a low-rank structure, and the tensor-train decomposition \cite{oseledets2011tensor}, also known as the matrix product state representation \cite{Perez-Garcia2007MPS}, can exploit the inherent low-rank structure to mitigate the curse of dimensionality in many practical scenarios.

%The tensor-train (TT) decomposition has found widespread applications in areas such as machine learning \cite{yang2017tensor,novikov2015tensorizing,CHEN2022108337}, quantum many-body physics \cite{Schollwock2011wave}, signal processing \cite{miron2020tensor, Sidiropoulos2017}, data completion \cite{xu2023tensor,Steinlechner2016, Cai2022provable}, and numerical partial differential equations \cite{Dolgov2021,Richter2021Solving}. For instance, the TT decomposition has been employed to compress deep neural networks while maintaining competitive performance \cite{yang2017tensor,novikov2015tensorizing}. 
%In quantum physics, the TT format enables efficient representation and manipulation of wave functions and operators, thereby significantly reducing computational demands \cite{Schollwock2011wave}. Over the past decade, the TT model has also been adopted for certain high-dimensional data completion problems \cite{Steinlechner2016, Cai2022provable}. Besides, TT-based methods are increasingly applied in solving high-dimensional partial differential equations, where traditional methods often become computationally prohibitive \cite{Dolgov2021,Richter2021Solving}, to name a few.

The tensor-train (TT) decomposition has found widespread applications in areas such as machine learning \cite{novikov2015tensorizing, CHEN2022108337}, quantum many-body physics \cite{Schollwock2011wave}, signal processing \cite{miron2020tensor, Sidiropoulos2017,wang2025tensor}, data completion \cite{xu2023tensor,Steinlechner2016, Cai2022provable}, and numerical partial differential equations \cite{Dolgov2021,Richter2021Solving}. For instance, the TT decomposition has been employed to compress deep neural networks while maintaining competitive performance \cite{novikov2015tensorizing}. In quantum physics, the TT format enables efficient representation and manipulation of wave functions and operators, thereby significantly reducing computational demands \cite{Schollwock2011wave}. 
%Over the past decade, the TT model has also been adopted for certain high-dimensional data completion problems \cite{Steinlechner2016, Cai2022provable}. 
Besides, TT-based methods are increasingly applied in solving high-dimensional partial differential equations, where traditional methods often become computationally prohibitive \cite{Dolgov2021,Richter2021Solving}, to name a few.

The research on efficient algorithms for constructing TT decomposition of large-scale tensors remains an active area. For a $d$th-order tensor, the TT-SVD algorithm \cite{oseledets2011tensor} sequentially utilizes $d-1$ singular value decompositions to extract core tensors one by one, which serves as a foundational approach to decomposing a higher-order tensor into the TT format. TT-cross approximation \cite{oseledets2010tt} replaces the full SVD computation in the TT-SVD algorithm with skeleton decomposition, making it applicable to large-scale tensors that contain missing entries or require data-dependent computation. The alternating least squares method, as an optimization method, computes the TT decomposition by solving a sequence of least squares problems to iteratively refine the TT cores \cite{Holtz2012the}. These optimization methods typically require TT ranks to be specified in advance, and their convergence behavior is dependent on the initial values and the conditioning of the underlying problem. When high precision is required, one strategy is to use the quasi-optimal solution obtained from algebraic methods such as TT-SVD as an initial guess, followed by refinement through optimization-based approaches. In recent years, experts have also proposed parallel algorithms to reduce the runtime for computing TT decomposition; their subroutines still rely on the matrix SVD \cite{zniyed2020tt, shi2023parallel}. %Based on the theory of randomized numerical linear algebra, randomized algorithms \cite{Che2019randomized} for TT decomposition have been proposed to reduce the computational complexity under the framework of the TT-SVD algorithm.

The TT-SVD algorithm has attracted substantial research interest, owing to its numerical stability, finite-step computation, and controllable accuracy. It has been widely applied to real-world applications and has served as the foundation for many subsequent algorithms for TT decomposition. For instance, in high-dimensional data completion via the TT model, Riemannian optimization methods are often employed, which use the TT-SVD algorithm at each iteration to retract the updated tensor onto the manifold \cite{Steinlechner2016, Cai2022provable}. For the joint estimation of two-dimensional direction-of-departure and direction-of-arrival in MIMO radar systems, recent studies have introduced TT-based parameter estimation frameworks, where the TT-SVD algorithm is employed as a key step to decompose the tensors constructed from signal data \cite{xie2025coupled, xie2025higher}. These methods demonstrate improvements in preserving the intrinsic multidimensional structure and estimation accuracy. In recent years, randomized algorithms have also been developed within the TT-SVD framework to reduce computational complexity \cite{huber2018randomized, che2026efficient}.

However, in many applications, the SVD is computationally demanding and difficult to update. These limitations have motivated the development of the rank-revealing two-sided orthogonal decomposition, also called the UTV decomposition \cite{Stewart1992an}. The UTV decomposition provides approximately the same numerical rank and subspace information as SVD but at a lower computational cost \cite{Stewart1992an,Fierro1997low,Fierro1995bounding}. Furthermore, the UTV decomposition can be efficiently implemented on modern hardware architectures by leveraging block matrix operations that are amenable to parallel processing \cite{Martinsson2019rand}. Due to these virtues, it has become a widely used alternative to SVD in applications such as image denoising and signal processing \cite{yang1995projection, Kaloorazi2018Com}. Recently, the UTV decomposition has also been extended to the decompositions of higher-order tensors. For example, 
a truncated multilinear UTV decomposition \cite{Vandecappelle2022from} is introduced and numerically validated as a viable practical alternative to the truncated multilinear SVD \cite{DeLathauwer2000a,Vannieuwenhoven2012a}; the t-UTV decomposition \cite{che2022efficient} is proposed for third-order tensors as a substitute for the t-SVD \cite{kilmer2011factorization} based on the tensor t-product. UTV decomposition for quaternion matrices and tensors, together with their randomized algorithms, have also been developed \cite{YANG2025111580}. 

In this article, we investigate whether the UTV decomposition can be leveraged to construct the TT decomposition as an economic alternative to the TT-SVD. %The goal is to save computational costs by utilizing the advantages of UTV decomposition over SVD for large-scale unfolding matrices. 
This new approach, termed TT-UTV, is designed to operate under both fixed TT-ranks and fixed precision. There exist two types of TT-UTV algorithms corresponding to the URV and ULV cases, respectively. We provide the error analysis for these algorithms and demonstrate their effectiveness through various numerical experiments, including video recovery and magnetic resonance imaging data completion.

The rest of this paper is organized as follows. Section \ref{sec: pre} reviews the necessary preliminary notions. In Section \ref{sec: 3}, we introduce the TT-ULV algorithm in a left-to-right sweep and the TT-URV algorithm in a right-to-left sweep, respectively. We further derive rigorous error bounds for them, which explain how the UTV decompositions are used to achieve accurate and efficient TT decomposition; their computational cost is also analyzed. In Section \ref{sec: 4}, we evaluate the performance of the TT-UTV algorithm against several existing TT decomposition algorithms in numerical experiments. The conclusions are drawn in Section \ref{sec: 5}.

\section{Preliminaries} 
\label{sec: pre}

\textit{Notations}. A $d$th-order tensor is a multidimensional array with $d$ free indices. Matrices are second-order tensors, and vectors are first-order tensors. We mainly inherit the usage of notations from the comprehensive work of Kolda and Bader \cite{kolda2009tensor}. The calligraphic capital letters such as $\mathcal{A}$, boldface capital letters such as $\mathbf{A}$, and boldface lowercase letters such as $\mathbf{a}$ are used to denote tensors, matrices, and vectors, respectively. $\mathbb{R}^{I_1\times \cdots \times I_d}$ denotes the set of all $d$th-order tensors of dimensions $(I_1, I_2, \dots, I_d)$ over the real number field $\mathbb{R}$, where $I_k$ is the dimension of the $k$th mode of the tensor. The entries of $\mathcal{A} \in \mathbb{R}^{I_1 \times \cdots \times I_d}$ are accessed by $\mathcal{A}_{i_1i_2 \cdots i_d}$ for $1 \leq i_k \leq I_k$. Table~\ref{table: Nomenclature} shows the notations frequently used in this article.

\begin{table}[htbp!]
\centering
\caption{Nomenclature}
\label{table: Nomenclature}
\begin{tabular}{|l|l|}
\hline $\mathcal{A} \in \mathbb{R}^{I_1\times \cdots \times I_d}$ & Input tensor of the algorithms \\
\hline $\widehat{\mathcal{A}}\in \mathbb{R}^{I_1\times \cdots \times I_d}$ & Output tensor of the algorithms \\
\hline $\left\{r_k\right\}_{k=0}^d$ & The TT-ranks of tensor $\mathcal{A}$ \\
\hline $\mathcal{G}^{(k)} \in \mathbb{R}^{r_{k-1}\times I_k \times r_k}$ & The $k$th TT core of tensor $\mathcal{A}$ \\
% \hline $\mathbf{C}$ & Temporary matrix during the computation \\
% \hline $\mathcal{A}^{(k)} \in \mathbb{R}^{r_{k-1}I_k\times I_{k+1} \times \cdots \times I_d}$ & The tensor of order $d-k+2$ reshaped from $\mathbf{C}$ \\
% \hline $\widetilde{\mathcal{A}}^{(k)}\in \mathbb{R}^{r_{k-1}\times I_k \times \cdots \times I_d}$ & The tensor of order $d-k+1$ reshaped from $\mathbf{C}$ \\
\hline $\mathbf{U}^{(k)} \in \mathbb{R}^{r_{k-1}I_k\times r_k}$ & The matrix to form the $k$th TT core $\mathcal{G}^{(k)}$ \\
\hline $\mathbf{U}_k, \mathbf{V}_k, k=1,2$ & Partitioned matrices of $\mathbf{U}$ and $\mathbf{V}$ in UTV decomposition \\
\hline $\mathbf{L}_{ij}, \mathbf{R}_{ij}, i,j=1,2$ & Partitioned matrices of $\mathbf{T}$ in UTV decomposition \\
\hline $\mathbf{A}_k \in \mathbb{R}^{(I_1\cdots I_k )\times (I_{k+1}\cdots I_d)}$ & The $k$th unfolding matrix of tensor $\mathcal{A}$ (Definition \ref{Def: Unfoldmat}) \\
\hline $\mathbf{E}$ & The remaining part of the truncated UTV decomposition \\
\hline $\mathbf{I}_m$ & Identity matrix of size $m\times m$ \\
\hline $\mathbf{O}_{m\times n}$ & Zero matrix of size $m\times n$ \\
\hline
\end{tabular}
\end{table}

\subsection{Tensor basics}

% Next, we present some important notions and results that are used in the numerical computations or theoretical error analysis.

The TT decomposition represents a higher-order tensor $\mathcal{A} \in \mathbb{R}^{I_1\times \cdots \times I_d}$ as the contractions of small factor tensors
\begin{equation}\label{TT}
	\mathcal{A}_{i_1\cdots i_d} = \sum_{\alpha_0=1}^{r_0}\sum_{\alpha_1=1}^{r_1}  \cdots \sum_{\alpha_d=1}^{r_d}  \mathcal{G}^{(1)}_{\alpha_0i_1\alpha_1}\mathcal{G}^{(2)}_{\alpha_1 i_2 \alpha_2} \cdots \mathcal{G}^{(d)}_{\alpha_{d-1} i_d\alpha_d}, 
\end{equation}
where $r_0=r_d=1$, third-order tensors $\mathcal{G}^{(k)} \in \mathbb{C}^{r_{k-1}\times I_k \times r_k}$ for $k=1,2,\ldots,d$ are called TT-cores, and $(r_1, r_2,\ldots, r_{d-1})$ are TT-ranks. Fig.~\ref{Graph: TT} is a graphical representation of the TT decomposition of a $d$th-order tensor, with core tensors depicted as blue blocks and indices represented as connecting lines. It looks like a train, as its name implies. By virtue of the TT decomposition, the tensor $\mathcal{A} \in \mathbb{C}^{I_1\times \cdots \times I_k}$ can be compressed into $ {\textstyle \sum_{k=1}^{d}} r_{k-1}I_kr_k $ parameters, which become linearly dependent on the tensor order $d$. Lower TT-ranks imply lower memory consumption and computational costs. 
\begin{figure}[htbp!]
	\centering
	\includegraphics[width=0.8\textwidth]{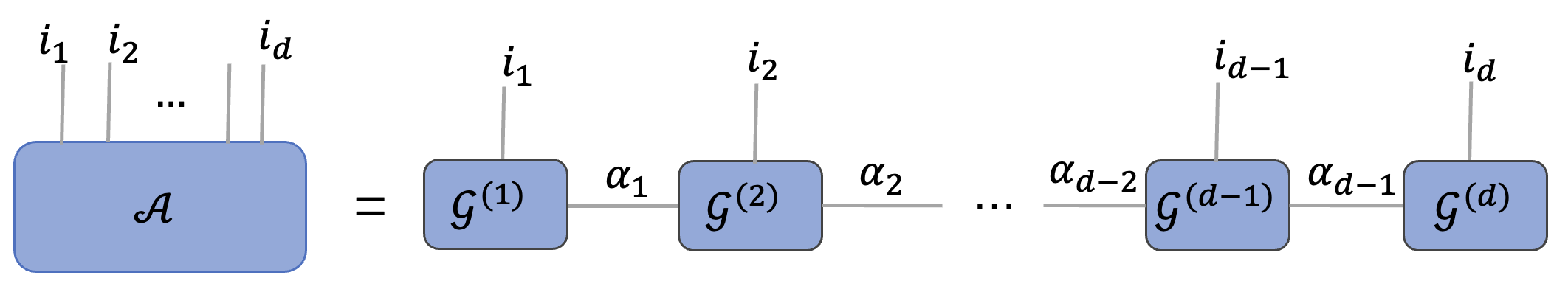}
	\caption{Graphical illustration of the TT decomposition for a $d$th-order tensor.}
	\label{Graph: TT}
\end{figure}

Tensor unfolding, the process of rearranging the elements of a tensor into a matrix or vector, is a frequently used operation in tensor analysis and numerical computations \cite{kolda2009tensor}. The reverse lexicographic ordering (as the `reshape' command in MATLAB) is one of the most commonly used approaches for integrating multiple indices into a single long index. Concretely, for the tensor dimensions $\mathbf{I} = \left( I_1,\dots, I_d \right)$ and a corresponding index vector $\mathbf{i} = \left(i_1,\dots, i_d\right)$, the reverse lexicographic order map
\begin{equation*} %\label{eq: ivec}
	\mathrm{ivec}(\mathbf{i}, \mathbf{I}) := i_1+\sum_{k=2}^{d} (i_k-1)\prod_{j=1}^{k-1}I_j.   
\end{equation*}

\begin{definition}[The $k$th unfolding matrix \cite{oseledets2011tensor}]
\label{Def: Unfoldmat}
	For a $d$th-order tensor $\mathcal{A} \in \mathbb{R}^{I_1 \times \cdots \times I_d}$, the matrix $\mathbf{A}_k \in \mathbb{R}^{(I_1 \cdots I_k) \times (I_{k+1}\cdots I_d)}$ is called the $k$th unfolding matrix of $\mathcal{A}$ such that
	\begin{equation*}
		(\mathbf{A}_k)_{i,j} = \mathcal{A}_{i_1\cdots i_k i_{k+1}\cdots i_d},
	\end{equation*}
	where $i = \mathrm{ivec}\left(( i_1,\dots,i_k ), (I_1,\dots, I_k )\right)$ and $j = \mathrm{ivec}\left( (i_{k+1},\dots,i_d ), ( I_{k+1},\dots, I_d)\right)$. 
    % That is to say, the first $k$ indices and the last $d-k$ indices of $\mathcal{A}$ are fused into the row and column indices of the unfolding matrix $\mathbf{A}_k$, respectively, following the reverse lexicographic order $\rm ivec$ in Eq. \eqref{eq: ivec}.
\end{definition}

\begin{definition}[Orthogonal TT-cores \cite{oseledets2011tensor}]
	The TT-cores $\{\mathcal{G}^{(k)}, k=1,\dots, d\}$ in Eq.  \eqref{TT} are called left-orthogonal $ \mathbf{G}^{(k)\top}_2\mathbf{G}^{(k)}_2 = \mathbf{I}_{r_k}$ if for $k=1,\dots, d-1$;  
	they are right-orthogonal if  $ \mathbf{G}^{(k)}_1\mathbf{G}^{(k)\top}_1 = \mathbf{I}_{r_{k-1}}, $ for $k=2,\dots, d$.
    Here $\mathbf{G}^{(k)}_1 \in \mathbb{R}^{r_{k-1} \times r_kI_k }$ and $\mathbf{G}^{(k)}_2 \in \mathbb{R}^{r_{k-1}I_k \times r_k}$ are the first and second unfolding matrices of the TT-core $\mathcal{G}^{(k)}$, respectively.
\end{definition}
The TT-SVD algorithm computes the TT format \eqref{TT} with left-orthogonal cores by a single left-to-right sweep or the right-orthogonal cores by a right-to-left sweep; more details can be found in \cite{oseledets2011tensor}.

\begin{definition}[Mode product \cite{kolda2009tensor}]
	The mode-$k$ product of $\mathcal{S} \in \mathbb{R}^{I_1 \times \cdots \times I_d}$ with a matrix $\mathbf{Q} \in \mathbb{R}^{J \times I_k}$ is an $I_1 \times \cdots \times I_{k-1} \times J \times I_{k+1} \times \cdots \times I_d$ tensor, given by
	\[ \left(\mathcal{S} \times _k \mathbf{Q} \right)_{i_1\cdots i_{k-1}ji_k\cdots i_d} = \sum_{i_k=1}^{I_k} \mathcal{S}_{i_1\cdots i_{k-1}i_k i_{k+1}\cdots i_d} \mathbf{Q}_{ji_k}. \]
\end{definition}

Apart from the TT decomposition, the Tucker format \cite{tucker1966some} is another compact form for large-scale tensors, which represents a $d$th-order tensor via the mode products of a smaller-scaled core tensor and $d$ factor matrices: 
\begin{align} \label{Tucker}
	\mathcal{A} = \mathcal{S} \times_1 \mathbf{Q}^{(1)} \times_2 \cdots \times _d \mathbf{Q}^{(d)}. 
\end{align}
Directly from these definitions, we can derive the following proposition.
\begin{proposition}\label{Pro: 2.1}
	Let $\mathcal{A}$ be given in the Tucker format \eqref{Tucker}. Then the $k$th unfolding matrices $\mathbf{A}_k$ and $\mathbf{S}_k$ are related by the matrix equation
	\[ \mathbf{A}_k = \left( \mathbf{Q}^{(k)} \otimes \cdots \otimes \mathbf{Q}^{(1)} \right) \mathbf{S}_k \left( \mathbf{Q}^{(d)} \otimes \cdots \otimes \mathbf{Q}^{(k+1)} \right)^{\top},  \]
	where $\otimes$ denotes the matrix Kronecker product.
\end{proposition}

%For $\mathbf{A} \in \mathbb{R}^{m\times n}$ and $\mathbf{B} \in \mathbb{R}^{p\times q}$, the Kronecker product $\mathbf{A}\otimes \mathbf{B}$ is the $mp\times nq$ block matrix
%\begin{equation*}
%	\mathbf{A} \otimes \mathbf{B}=\left[\begin{array}{ccc}
%		a_{11} \mathbf{B} & \cdots & a_{1 n} \mathbf{B} \\
%		\vdots & \ddots & \vdots \\
%		a_{m 1} \mathbf{B} & \cdots & a_{m n} \mathbf{B}
%	\end{array}\right].
%\end{equation*}

\begin{proposition}[\cite{golub2013matrix}] \label{Pro: MatrixKron}
	Let $\mathbf{A}\in \mathbb{R}^{m\times n}, \mathbf{B} \in \mathbb{R}^{p\times q}$,  $\mathbf{C}\in \mathbb{R}^{n\times l}$, and $\mathbf{D}\in \mathbb{R}^{q \times r}$. Then $(\mathbf{A}\otimes \mathbf{B})^{\top} = \mathbf{A}^{\top} \otimes \mathbf{B}^{\top}$, and $(\mathbf{A}\otimes \mathbf{B}) (\mathbf{C}\otimes \mathbf{D}) = (\mathbf{A}\mathbf{C})\otimes (\mathbf{B} \mathbf{D})$.
\end{proposition}

\subsection{UTV decompositions}
For the matrix $\mathbf{A} \in \mathbb{R}^{m\times n}$, its UTV decomposition takes the form $\mathbf{A} = \mathbf{UTV}^{\rm \top}$, where $\mathbf{U}\in \mathbb{R}^{m\times n}$ and $\mathbf{V} \in \mathbb{R}^{n\times n}$ have orthonormal columns, and $\mathbf{T}$ is a triangular matrix. There are two different cases of UTV decomposition corresponding to the structure of the middle matrix $\mathbf{T}$. If $\mathbf{T}$ is lower triangular, the decomposition is called ULV decomposition:
\begin{equation} \label{eq: ulv}
	\mathbf{A} = \mathbf{U}\begin{bmatrix}
		\mathbf{L}_{11} & \mathbf{O}_{r\times (n-r)} \\
		\mathbf{L}_{21} & \mathbf{L}_{22}
	\end{bmatrix} \mathbf{V}^{\top}.
\end{equation}
Here $\mathbf{L}_{11}$ and $\mathbf{L}_{22}$ are lower triangular, $\mathbf{L}_{21} \in \mathbb{R}^{(m-r)\times r}$ with small entries, and $\mathbf{O}$ denotes the zero matrix whose elements are all 0. Moreover, suppose the numerical rank of $\mathbf{A}$ is $r$, then the ULV decomposition is said to be rank-revealing if
\begin{equation} \label{rrULV}
	\sigma_{\rm min}(\mathbf{L}_{11}) = O(\sigma_r) \quad \mathrm{and} \quad \left\| \begin{bmatrix}
		\mathbf{L}_{21}& \mathbf{L}_{22}
	\end{bmatrix} \right\|_2 = O(\sigma_{r+1}). 
\end{equation}
If the middle matrix $\mathbf{T}$ is upper-triangular, then the decomposition called URV decomposition, which takes the form
\begin{equation} \label{eq: urv}
	\mathbf{A} = \mathbf{U}\begin{bmatrix}
		\mathbf{R}_{11} & \mathbf{R}_{12} \\
		\mathbf{O}_{(m-r)\times r}  & \mathbf{R}_{22}
	\end{bmatrix} \mathbf{V}^{\top},
\end{equation}
where $\mathbf{R}_{11}\in \mathbb{R}^{r\times r}$, $\mathbf{R}_{22} \in \mathbb{R}^{(m-r)\times (n-r)}$ are upper triangular, and $\mathbf{R}_{12} \in \mathbb{R}^{r \times (n-r)}$ contains small entries. The URV decomposition is said to be rank-revealing if
\begin{equation} \label{rrURV}
	\sigma_{\rm min}(\mathbf{R}_{11}) = O(\sigma_r) \quad \mathrm{and} \quad \left\| \begin{bmatrix}
		\mathbf{R}_{12}^{\top}& \mathbf{R}_{22}^{\top}
	\end{bmatrix} \right\|_2 = O(\sigma_{r+1}).
\end{equation}

The smaller the norm of the off-diagonal block, the better the truncated low-rank approximations in Eqs. \eqref{eq: ulv} and \eqref{eq: urv}. In fact, the UTV decomposition of a matrix is not unique, and various matrix decompositions proposed in the literature can be regarded as UTV variants, such as the rank-revealing QR decomposition \cite{CHAN198767}, the pivoted QLP decomposition \cite{stewart1999qlp}, and the SVD, which is also a special case of UTV decomposition when the middle matrix is diagonal. There are many efficient algorithms for UTV decomposition. Stewart's foundational work \cite{Stewart1992an} introduced the algorithms for UTV decomposition that are suited for high-rank matrices, while Fierro et al. \cite{Fierro1997low,fierro1999utv} later proposed low-rank UTV algorithms and designed a Matlab toolbox "UTV Tools" for efficient rank-revealing factorizations. More recently, randomized algorithms for UTV decomposition have been developed \cite{Martinsson2019rand,kaloorazi2020efficient}. In practice, the UTV factorization often attains comparable accuracy at a lower computational cost than SVD.

\section{The proposed methods}
\label{sec: 3}
%In this section, we propose the TT-UTV algorithms computing the approximate TT-format by sequentially separating the cores through the UTV decompositions and demonstrate their feasibility by theoretical analysis. We discuss these algorithms in two cases, corresponding to the ULV and URV decompositions, respectively. 

In this section, we propose the TT-UTV algorithms for efficiently computing the approximate TT format by sequentially decomposing the cores via UTV factorizations. We discuss two variants of the algorithm, corresponding to the ULV and URV decompositions, respectively. The feasibility of the proposed methods is rigorously supported by theoretical analysis, based on which we recommend distinct update strategies for each case and further introduce fixed-precision algorithms.

\subsection{TT-ULV case}
If the required TT-ranks are fixed, the TT-ULV algorithm is shown in Algorithm~\ref{alg:TT-UTV}, which computes the approximate TT-format tensor with left-orthogonal cores in a left-to-right sweep.
\begin{algorithm}[!htpb]
	\renewcommand{\algorithmicrequire}{\textbf{Input:}}
	\renewcommand{\algorithmicensure}{\textbf{Output:}}
	\caption{TT-ULV algorithm for fixed TT-ranks (left-to-right sweep)}
	\label{alg:TT-UTV}
	\centering
	\begin{algorithmic}[1]
		\REQUIRE A $d$th-order tensor $\mathcal{A} \in \mathbb{R}^{I_1 \times \cdots \times I_d}$, and fixed TT-ranks $\{ r_1, \dots, r_{d-1} \}$, $r_0=r_d=1$. 
		\ENSURE Left-orthogonal TT-cores $\{\mathcal{G}^{(1)}, \dots, \mathcal{G}^{(d)}\}$ of the approximation $\widehat{\mathcal{A}}$ with TT-ranks $r_k$'s.
		
		\STATE Temporary matrix: $\mathbf{C}=\mathrm{reshape}(\mathcal{A}, \left[ I_1, I_2\cdots I_d \right])$.
		\STATE for $k=1$ to $d-1$ do
		\STATE \quad  $\mathbf{C}:=\operatorname{reshape}\left(\mathbf{C},\left[r_{k-1} I_k, I_{k+1}\cdots I_d\right]\right)$.
		\STATE \quad Compute the truncated rank-$r_k$ approximation from ULV decomposition $\mathbf{C}=\mathbf{U}_1 \mathbf{L}_{11} \mathbf{V}_1^{\top}+\mathbf{E}$, where $\mathbf{L}_{11}\in \mathbb{R}^{r_k\times r_k}$ is the leading term of $\mathbf{L}$, and $\mathbf{E}$ denotes the residual part in Eq. \eqref{eq: ulv}.
		\STATE \quad New core: $\mathcal{G}^{(k)}=\operatorname{reshape}\left(\mathbf{U}_1,\left[r_{k-1}, I_k, r_k\right]\right)$.
		\STATE \quad  $\mathbf{C}= \mathbf{L}_{11} \mathbf{V}_1^{\top}$.
		\STATE end for
		\STATE $\mathcal{G}^{(d)}=\mathbf{C}$.
		%\STATE Return tensor $\widehat{\mathcal{A}}$ in TT-format with cores $\mathcal{G}_1, \dots, \mathcal{G}_d$.
	\end{algorithmic}  
\end{algorithm} 

To facilitate the presentation of error analysis, we introduce some notations. The TT-ULV algorithm~\ref{alg:TT-UTV} adopts the left-to-right sweep (steps 2-7). In the $k$th iteration, step 3 computes a temporary matrix $\mathbf{C}\in \mathbb{R}^{r_{k-1}I_k\times I_{k+1}\cdots I_d}$; we denote its corresponding $(d-k+2)$th-order tensor by $\mathcal{A}^{(k)} = \mathrm{reshape}(\mathbf{C}, [r_{k-1},I_k,I_{k+1}, \cdots, I_d]) \in \mathbb{R}^{r_{k-1}\times I_k \times \cdots \times I_d}$, and the $(d-k+1)$th-order tensor by $\widetilde{\mathcal{A}}^{(k)} = \mathrm{reshape}($ $\mathbf{C}, [r_{k-1}I_k, I_{k+1},\cdots, I_d]) \in \mathbb{R}^{r_{k-1}I_k\times I_{k+1} \times \cdots \times I_d}$. Step 4 computes the column-orthogonal matrix $\mathbf{U}_1\in \mathbb{R}^{r_{k-1}I_k \times r_k}$, denoted by $\mathbf{U}^{(k)}$ in the $k$th iteration. With the notations introduced, the whole process generates temporary tensors $\{ \mathcal{A}^{(k)}\}_{k=1}^{d-1}$, $\{ \widetilde{\mathcal{A}}^{(k)}\}_{k=1}^{d-1}$, and column-orthogonal matrices $\{ \mathbf{U}^{(k)} \}_{k=1}^{d-1}$ that give rise to the TT-cores. 

\begin{proposition}\label{pro: 3.1}
	The unfolding matrices of tensors $\{ \mathcal{A}^{(k)}\}_{k=1}^{d-1}$ satisfy the following equations
	\begin{align}\label{eq: relations}
		\mathbf{A}^{(k+1)}_l = \left(\mathbf{I}_{I_{k+l-1}} \otimes \cdots \otimes \mathbf{I}_{I_{k+1}} \otimes \mathbf{U}^{(k)\top}\right)\mathbf{A}^{(k)}_{l+1}, 
	\end{align}
	for $l=1,\dots, d-k+2$ for all $k=1,\cdots, d-1$, where $\mathbf{A}_{l}^{(k)}$ is the $l$th unfolding matrix of tensor $\mathcal{A}^{(k)}\in \mathbb{R}^{r_{k-1}\times I_k \times I_{k+1}\times \cdots \times I_d}$.
\end{proposition}
\begin{proof}
	In the $k$th iteration of the TT-ULV algorithm~\ref{alg:TT-UTV}, the rank-$r_k$ truncated ULV decomposition in step 4 can be written as
	\[ \mathbf{A}_2^{(k)} = \mathbf{U}^{(k)}\mathbf{A}_1^{(k+1)}+\mathbf{E}, \]
	where $\mathbf{A}_2^{(k)}\in \mathbb{R}^{r_{k-1}I_k\times I_{k+1}\cdots I_d}, \mathbf{U}^{(k)}\in \mathbb{R}^{r_{k-1}I_k \times r_k}, \mathbf{A}_1^{(k+1)}\in \mathbb{R}^{r_k \times I_{k+1}\cdots I_d}$. According to the orthogonality $\mathbf{U}^{(k)\top}\mathbf{E} = \mathbf{O}$, we have
	\begin{align*} \label{eq: 3.1}
		\mathbf{A}_1^{(k+1)} = \mathbf{U}^{(k)\top}\mathbf{A}_2^{(k)}.
	\end{align*}
	From the definitions of the notations $\mathcal{A}^{(k)}$ and $\widetilde{\mathcal{A}}^{(k)}$, the $l$th unfolding  matrix of $\widetilde{\mathcal{A}}^{(k)}$ is exactly the $(l+1)$th unfolding matrix of $\mathcal{A}^{(k)}$, i.e., $\mathbf{A}^{(k)}_{l+1}$, for $l=1,\cdots, d-k+1$. Clearly, the above equation can be rewritten in tensor format
	\[ \mathcal{A}^{(k+1)} = \widetilde{\mathcal{A}}^{(k)} \times_1 \mathbf{U}^{(k)\top} \times_2 \mathbf{I}_{I_{k+1}} \times_3 \cdots \times_{d-k+1} \mathbf{I}_{I_d}. \]
	From the above tensor equation, the Eq. \eqref{eq: relations} holds by Proposition \ref{Pro: 2.1}.
\end{proof}

\begin{lemma}\label{lem: 3.1}
	Let $m \ge s \ge r $ be positive integers, $\mathbf{Q}\in \mathbb{R}^{s\times m}$ be a row-orthogonal matrix with unit-norm rows, $\mathbf{A}\in \mathbb{R}^{m \times n}$, and $\mathbf{B} = \mathbf{Q}\mathbf{A} \in \mathbb{R}^{s\times n}$. Then the rank-r truncated error of the ULV (or URV) decomposition of $\mathbf{B}$ can be smaller than that of $\mathbf{A}$.
\end{lemma}
\begin{proof}
	Suppose a rank-$r$ truncated ULV decomposition of $\mathbf{A}$ is 
	\[ \mathbf{A} = \widehat{\mathbf{U}}\widehat{\mathbf{L}}\widehat{\mathbf{V}}^{\top} + \mathbf{E},  \]
	and the truncated error $\| \mathbf{E} \|_F = \varepsilon$, where $\widehat{\mathbf{U}}\in \mathbb{R}^{m\times r}$. Then 
	\[ \mathbf{B} = \mathbf{Q}\widehat{\mathbf{U}}\widehat{\mathbf{L}}\widehat{\mathbf{V}}^{\top} + \mathbf{Q}\mathbf{E}. \]
	Since $\mathbf{Q}$ has orthogonal rows with unit norm, $\| \mathbf{Q} \|_2 =1$, then $\| \mathbf{Q}\mathbf{E} \|_F \le \| \mathbf{E} \|_F = \varepsilon$. For $\mathbf{Q}\widehat{\mathbf{U}} \in \mathbb{R}^{s\times r}$, we can perform the Gram-Schmidt orthogonalization process on its columns to obtain a decomposition (similar to the QR decomposition) $\mathbf{Q}\widehat{\mathbf{U}} = \widetilde{\mathbf{U}}\mathbf{L}$, where $\widetilde{\mathbf{U}}\in \mathbb{R}^{s\times r}$ has orthogonal columns with unit norm and $\mathbf{L}\in \mathbb{R}^{r \times r}$ is lower-triangular. Then $\widetilde{\mathbf{L}} = \mathbf{L}\widehat{\mathbf{L}}$ is still a lower-triangular matrix. Therefore,
	\[ \mathbf{B} = \widetilde{\mathbf{U}}\widetilde{\mathbf{L}}\widehat{\mathbf{V}}^{\top} + \mathbf{Q}\mathbf{E},  \]
	the leading term $\widetilde{\mathbf{U}}\widetilde{\mathbf{L}}\widehat{\mathbf{V}}^{\top}$ is actually a rank-$r$ truncated ULV decomposition of $\mathbf{B}$, and the residual part $\| \mathbf{QE}\|_F \le \varepsilon$. Thus, it finds a rank-$r$ truncated ULV decomposition for $\mathbf{B}$ such that the approximate error is smaller than that of $\mathbf{A}$. For the URV case, the above analysis similarly holds.
\end{proof}

\begin{remark} \label{rmk: 3.1}
	According to Proposition \ref{Pro: MatrixKron}, the matrices $\left(\mathbf{I}_{I_{k+l-1}} \otimes \cdots \otimes \mathbf{I}_{I_{k+1}} \otimes \mathbf{U}^{(k)\top}\right)$ in Eq. \eqref{eq: relations} are row-orthogonal with unit-norm rows. 
\end{remark}

Next, the error analysis for the TT-ULV algorithm~\ref{alg:TT-UTV} can be derived as shown in Theorem \ref{thm: 3.1} and Corollary \ref{cor: 3.1}, which can be seen as a generalization of the results for TT-SVD \cite{oseledets2011tensor}.

\begin{theorem}[TT-ULV case] \label{thm: 3.1}
	Given a tensor $\mathcal{A} \in \mathbb{R}^{I_1\times \cdots \times I_d}$. Assume that the rank-$r_k$ approximation error of its $k$th unfolding matrix $\mathbf{A}_k \in \mathbb{R}^{(I_1\cdots I_k)\times (I_{k+1}\cdots I_d)}$ from the ULV decomposition is $\varepsilon_k$, that is,
	\begin{equation}\label{eq: ULVcase}
		\mathbf{A}_k = \begin{bmatrix}
			\mathbf{U}^{(k)}_{1} & \mathbf{U}^{(k)}_{2}
		\end{bmatrix} \begin{bmatrix}
			\mathbf{L}^{(k)}_{11} & \mathbf{O} \\
			\mathbf{L}^{(k)}_{21} & \mathbf{L}^{(k)}_{22}
		\end{bmatrix} \begin{bmatrix}
			\mathbf{V}^{(k)\top}_{1} \\
			\mathbf{V}^{(k)\top}_{2}
		\end{bmatrix} = \mathbf{U}^{(k)}_{1}\mathbf{L}^{(k)}_{11}\mathbf{V}^{(k)\top}_{1} + \mathbf{E}^{(k)},
	\end{equation}
	where $\mathbf{L}^{(k)}_{11}\in \mathbb{R}^{r_k \times r_k}$, $\mathbf{E}^{(k)} = \mathbf{U}^{(k)}_{2}(\mathbf{L}^{(k)}_{21}\mathbf{V}^{(k)\top}_{1}+\mathbf{L}^{(k)}_{22}\mathbf{V}^{(k)\top}_{2})$ and $\| \mathbf{E}^{(k)} \|_F = \varepsilon_k$ for all $k=1,\dots, d-1$. Then the approximate tensor $\widehat{\mathcal{A}}$ with TT-ranks not higher than $\{r_k\}_0^d$, computed via the TT-ULV algorithm~\ref{alg:TT-UTV} in a left-to-right sweep, meets
	\[ \| \mathcal{A}-\widehat{\mathcal{A}} \|_F \le \sqrt{\sum_{k=1}^{d-1} \varepsilon_k^2}.\]
\end{theorem}
\begin{proof}
	In the TT-ULV algorithm~\ref{alg:TT-UTV}, it yields the first TT-core from the ULV decomposition of the first unfolding matrix
	\[\mathbf{A}_1 = \mathbf{U}_{1}^{(1)}\mathbf{A}_2^{(2)} + \mathbf{E}^{(1)}, \]
	where $ \mathbf{U}^{(1)}_{1} \in \mathbb{R}^{I_1 \times r_1}$ gives the first TT-core, and $\mathbf{A}_2^{(2)} = \mathbf{L}^{(1)}_{11}\mathbf{V}^{(1)\top}_{1} \in \mathbb{R}^{r_1 \times (I_2\cdots I_d)}$ is used to produce the subsequent TT-cores. Note that $\mathbf{A}_2^{(2)}$ is the second unfolding matrix of the tensor $\mathcal{A}^{(2)}$. Based on the column orthogonality, it follows that $\mathbf{U}_1^{(1)\top} \mathbf{E}^{(1)} = \mathbf{O}$. As we can see, the subsequent processes are to approximate the $(d-1)$th-order tensor $\widetilde{\mathcal{A}}^{(2)} = \mathrm{reshape}(\mathbf{A}_2^{(2)}, [r_1I_2, I_3, \cdots, I_d])$ folded from the matrix $\mathbf{A}_2^{(2)}$. We denote by $\mathcal{B}\in \mathbb{R}^{r_1I_2 \times I_3\times \cdots \times I_d}$ the approximation of $\widetilde{\mathcal{A}}^{(2)}$ computed by the algorithm, and $\mathbf{B} = \mathrm{reshape}(\mathcal{B}, [r_1, (I_2\cdots I_d)])$. Then, the approximate error of the output tensor $\widehat{\mathcal{A}}$ of the TT-ULV algorithm is
	\begin{align} \label{Eq: err}
		\begin{split}
			\| \mathcal{A}-\widehat{\mathcal{A}} \|_F^2 & = \| \mathbf{A}_1 - \mathbf{U}_1^{(1)}\mathbf{B} \|_F^2 \\
			& = \| \mathbf{A}_1 - \mathbf{U}_1^{(1)}\mathbf{A}_2^{(2)} + \mathbf{U}_1^{(1)}\mathbf{A}_2^{(2)} -  \mathbf{U}_1^{(1)}\mathbf{B} \|_F^2 \\
			& \le \| \mathbf{E}^{(1)} \|_F^2 + \| \mathbf{A}_2^{(2)} - \mathbf{B} \|_F^2 \\
			& = \varepsilon_1^2 + \| \widetilde{\mathcal{A}}^{(2)} - \mathcal{B} \|_F^2,
		\end{split}
	\end{align}
	where the third inequality is derived from the orthogonality of $\mathbf{U}_1^{(1)\top}\mathbf{E}^{(1)}=\mathbf{O}$ and $\| \mathbf{U}_1^{(1)} \|_2 = 1$. Therefore, it reduces to the estimation of the approximate error for the $(d-1)$th-order tensor $\widetilde{\mathcal{A}}^{(2)}$ by TT-ULV algorithm. As explained in Proposition \ref{pro: 3.1}, Lemma \ref{lem: 3.1}, and Remark \ref{rmk: 3.1}, the truncated error of rank-$r_k$ approximation from the ULV decomposition of the $(k-1)$th unfolding matrix of $\widetilde{\mathcal{A}}^{(2)}$ can be smaller than the approximate error $\varepsilon_k$ of the rank-$r_k$ truncated ULV decomposition for $\mathbf{A}_k$, for all $k=2,\dots,d-1$. Hence, based on the above analysis and by the method of mathematical induction, we get

	\[ \| \widetilde{\mathcal{A}}_2 - \mathcal{B} \|_F^2 \le \sum_{k=2}^{d-1} \varepsilon_k^2. \]
	Combine it with the Eq. \eqref{Eq: err}, the theorem is proved.
\end{proof}

As established by Theorem \ref{thm: 3.1}, the approximation error of the TT-ULV algorithm is bounded by the low-rank approximation errors of the ULV decompositions applied to its unfolding matrices. This result also reveals that the low-rankness of these unfolding matrices directly implies a low-rank TT structure. Meanwhile, the above error analysis also illustrates how the final approximation error of the TT-UTV algorithm~\ref{alg:TT-UTV} propagates from the local errors incurred at each step of the sequential truncated UTV approximations.

%As established by the above theorem, the approximation error of the TT-ULV algorithm is bounded by the low-rank approximation error of the ULV decomposition for the unfolding matrices. Conversely, the low-rankness of these unfolding matrices implies a low-rank TT structure. From the analysis of Theorem \ref{thm: 3.1}, we can also derive the approximate error of the TT-UTV algorithm~\ref{alg:TT-UTV} from the local errors of each step in the truncated UTV decompositions.

\begin{corollary} \label{cor: 3.1}
	For the TT-ULV algorithm~\ref{alg:TT-UTV} in the left-to-right sweep, if the error of the $k$th truncated ULV decomposition in step 4 is $\varepsilon_k$, then the total error of the computed TT-approximation tensor $\widehat{\mathcal{A}}$ satisfies
	\[ \| \mathcal{A}-\widehat{\mathcal{A}} \|_F \le \sqrt{\sum_{k=1}^{d-1} \varepsilon_k^2}.\]
\end{corollary}

%Since $\mathbf{E} = \mathbf{U}_2\begin{bmatrix}
	%	\mathbf{L}_{21} &  \mathbf{L}_{21}
	%	\end{bmatrix} \mathbf{V}_2^{\top} = \mathbf{C} - \mathbf{U}_1\mathbf{L}_{11}\mathbf{V}_1^{\top}$, then $\| \mathbf{E} \|_F = \| \begin{bmatrix}
	%	\mathbf{L}_{21} &  \mathbf{L}_{21}
	%\end{bmatrix} \|_F$. 
	%Thereby, we only need to compute the Frobenius norms of row vectors of the intermediate matrix $\mathbf{L}$ from the ULV decomposition.  Concretely,
	%the cumulative sum vector $\mathbf{c}$ of these norms is computed in descending order of the row indices. Then the truncated rank $r_k$ can be determined as the smallest integer for which $c(n-r_k)\le \varepsilon_k$, where $n$ denote the length of vector $\mathbf{c}$, i.e., the number of rows in the middle matrix $\mathbf{L}\in \mathbb{R}^{n\times m}$. This process has a computational complexity of $O(mn)$, which is substantially lower than that of the preceding ULV decomposition.}
\begin{algorithm}[!htpb]
	\renewcommand{\algorithmicrequire}{\textbf{Input:}}
	\renewcommand{\algorithmicensure}{\textbf{Output:}}
	\caption{TT-ULV algorithm for prescribed accuracy (left-to-right sweep)}
	\label{alg:TT-UTVver2}
	\centering
	\begin{algorithmic}[1]
		\REQUIRE A $d$th-order tensor $\mathcal{A}$, and prescribed relative error tolerance $\varepsilon$, and the error weights $\{w_k\}_1^{d-1}$ with $\sum_{k=1}^{d-1} w_k^2 = 1$.
		\ENSURE Left-orthogonal TT-cores $\mathcal{G}^{(1)}, \dots, \mathcal{G}^{(d)}$ of the approximation tensor $\widehat{\mathcal{A}}$ with TT-ranks $r_k$ not higher than the ULV truncated ranks of the unfolding matrices $\mathbf{A}_k$ such that the truncated error $\varepsilon_k=w_k \varepsilon \|\mathcal{A}\|_F$. Then from Corollary \ref{cor: 3.1}, the computed TT-approximation tensor satisfies
		\[\|\mathcal{A}-\widehat{\mathcal{A}}\|_F \leq \varepsilon\|\mathcal{A}\|_F.\]
		
		\STATE  Compute truncation parameters $\varepsilon_k=w_k \varepsilon \|\mathcal{A}\|_F$, set $r_0 =1$.
		\STATE Temporary matrix: $\mathbf{C}=\mathrm{reshape}(\mathcal{A}, [I_1, I_2 \cdots I_d])$.
		\STATE for $k=1$ to $d-1$ do
		\STATE \quad $\mathbf{C}:=\mathrm{reshape}\left(\mathbf{C}, \left[r_{k-1} I_k,  I_{k+1}\cdots I_d\right]\right)$.
		\STATE \quad Compute truncated ULV decomposition $\mathbf{C}=\mathbf{U}_1 \mathbf{L}_{11}\mathbf{V}_1^{\top}+\mathbf{E}$, and determine the rank $r_k$ such that $\mathbf{L}_{11}\in \mathbb{R}^{r_k\times r_k}$ and $\|\mathbf{E}\|_F \leq \varepsilon_k$, the latter being equivalent to $\| \mathbf{L}_{11} \|_F^2 \ge \| \mathbf{C} \|_F^2-\varepsilon_k^2$. 
		\STATE \quad New core: $\mathcal{G}^{(k)}:=\operatorname{reshape}\left(\mathbf{U}_1, \left[r_{k-1}, I_k, r_k\right]\right)$.
		\STATE \quad $\mathbf{C}:=\mathbf{L}_{11}\mathbf{V}_1^{\top}$.
		\STATE end for
		\STATE $\mathcal{G}^{(d)}=\mathbf{C}$.
		%\STATE Return tensor $\widehat{\mathcal{A}}$ in TT-format with cores $\mathcal{G}_1, \dots, \mathcal{G}_d$.
	\end{algorithmic}  
\end{algorithm} 

Therefore, we also present the version of TT-ULV algorithm for prescribed accuracy in Algorithm~\ref{alg:TT-UTVver2} that can adaptively determine the TT-ranks.  In this algorithm, if equal error weights are used, then $w_k = \frac{1}{\sqrt{d-1}}$.  It is worth noting that in Step 5, we need to determine the smallest rank $r_k$ by comparing the truncated error of the ULV decomposition with the given tolerance $\varepsilon_k$. In the ULV desomposition $\mathbf{C} = \mathbf{U}_1\mathbf{L}_{11}\mathbf{V}_1^{\top} + \mathbf{E}$, the orthgonolity of $\mathbf{U}_1^{\top}\mathbf{E} = \mathbf{O}$ implies that $\| \mathbf{E} \|_F^2 = \| \mathbf{C} \|_F^2 - \| \mathbf{L}_{11} \|_F^2$. Thereby, we only need to compute the Frobenius norms of $\mathbf{C}$ and the row vectors of $\mathbf{L}_{11}$. Then the truncated rank $r_k$ can be determined as the smallest integer for which $\| \mathbf{L}_{11} \|_F^2  \ge \| \mathbf{C} \|_F^2 - \varepsilon_k^2$. Once this condition is satisfied, the computation of the current matrix factorization can be stopped;  therefore, there is always no need to compute the full UTV decomposition. If $\mathbf{C}$ is of size $m\times n$, then this process has a computational complexity of $O(mn)$, which is substantially lower than that of the preceding ULV decomposition.

As we can see, an essential requisite for the above error analysis for Algorithm~\ref{alg:TT-UTV} is that in each iteration the column-orthogonal matrix $\mathbf{U}_1$ used to generate the TT-core is orthogonal to the residual part $\mathbf{E}$ of the truncated ULV decomposition. Corollary \ref{cor: 3.1} theoretically controls the accumulation of the truncated errors by a sharp bound. 

However, if the right-orthogonal TT-cores are needed, the TT-ULV algorithm should adopt a right-to-left sweep, as given in Algorithm~\ref{alg:TT-ULV2}. Please refer to \ref{sec: appendix}, where we explain that this algorithm is not recommended for computing the right-orthogonal cores in the right-to-left sweep, especially for the objective of computing a fixed-precision TT decomposition. Interestingly, the URV case is precisely the opposite, which is more suitable to compute the TT format in a right-to-left sweep. 

\subsection{TT-URV case} \label{sec: 3.2}
Recall the URV decomposition \eqref{eq: urv}, 
$\mathbf{A} = \mathbf{U}_1\mathbf{R}_{11}\mathbf{V}_1^{\top}+(\mathbf{U}_1\mathbf{R}_{12}+\mathbf{U}_2\mathbf{R}_{22})\mathbf{V}_2^{\top}$, the matrix $\mathbf{V}_1$ is orthogonal to the residual part $(\mathbf{U}_1\mathbf{R}_{12}+\mathbf{U}_2\mathbf{R}_{22})\mathbf{V}_2^{\top}$. Wherefore, using the URV decomposition to compute the TT format is suitable for the right-to-left sweep. 

We formulate the TT-URV in Algorithm~\ref{alg:TT-URV} for fixed TT-ranks to replace Algorithm~\ref{alg:TT-ULV2} in a right-to-left sweep. The results similar to Proposition \ref{pro: 3.1}, Theorem \ref{thm: 3.1} and Corollary \ref{cor: 3.1} also hold for Algorithm~\ref{alg:TT-URV}. We omit the proofs.

% \begin{remark}
	%     From Theorems 3.1 and 3.2, to theoretically control the truncated errors in each UTV decomposition, the truncated parts $\mathbf{E}^{(k)}$ in the ULV \eqref{eq: ULVcase} and URV \eqref{eq: URVcase} cases are various. The TT-URV requires to compute an extra term in each step of the truncated URV decompositions over TT-ULV. 
	% \end{remark}

\begin{algorithm}[!htpb]
	\renewcommand{\algorithmicrequire}{\textbf{Input:}}
	\renewcommand{\algorithmicensure}{\textbf{Output:}}
	\caption{TT-URV algorithm for fixed TT-ranks (right-to-left sweep)}
	\label{alg:TT-URV}
	\centering
	\begin{algorithmic}[1]
		\REQUIRE A $d$th-order tensor $\mathcal{A} \in \mathbb{R}^{I_1 \times \cdots \times I_d}$, and fixed TT-ranks $\{ r_1, \dots, r_{d-1} \}$, $r_0=r_d=1$. 
		\ENSURE Right-orthogonal TT-cores $\mathcal{G}^{(1)}, \dots, \mathcal{G}^{(d)}$ of the approximation $\widehat{\mathcal{A}}$ with TT-ranks $r_k$'s.
		
		\STATE Temporary matrix: $\mathbf{C}=\mathrm{reshape}(\mathcal{A}, \left[ I_1\cdots I_{d-1}, I_d \right])$, i.e., the $(d-1)$th unfolding matrix $\mathbf{A}_{d-1}$.
		\STATE for $k=d$ to $2$ do
		\STATE \quad  $\mathbf{C}:=\operatorname{reshape}\left(\mathbf{C},\left[I_1\cdots I_{k-1}, I_{k}r_{k}\right]\right)$.
		\STATE \quad Compute the truncated rank-$r_{k-1}$ approximation from URV decomposition $\mathbf{C}=\mathbf{U}_1 \mathbf{R}_{11} \mathbf{V}_{11}^{\top}+\mathbf{E}$, where $\mathbf{U}_1 \mathbf{R}_{11} \mathbf{V}_1^{\top}$ with $\mathbf{R}_{11}\in \mathbb{R}^{r_{k-1}\times r_{k-1}}$ is the leading term, and $\mathbf{E}$ denotes the residual
		part in the URV decomposition \eqref{eq: urv}.
		\STATE \quad New core: $\mathcal{G}^{(k)}=\operatorname{reshape}\left(\mathbf{V}_1^{\top},\left[r_{k-1}, I_k, r_k\right]\right)$.
		\STATE \quad  $\mathbf{C}= \mathbf{U}_1 \mathbf{R}_{11}$.
		\STATE end for
		\STATE $\mathcal{G}^{(1)}=\mathbf{C}$.
		%\STATE Return tensor $\widehat{\mathcal{A}}$ in TT-format with cores $\mathcal{G}_1, \dots, \mathcal{G}_d$.
	\end{algorithmic}  
\end{algorithm}  

% \begin{proposition}\label{pro: 3.1}
	%     The the unfolding matrices of tensors $\{ \mathcal{A}^{(k)}\}_{k=1}^{d-1}$ satisfy the following equations
	%     \begin{align}\label{eq: relations}
		%         \mathbf{A}^{(k+1)}_l = \left(\mathbf{I}_{I_{k+l-1}} \otimes \cdots \otimes \mathbf{I}_{I_{k+1}} \otimes \mathbf{U}^{(k)\top}\right)\mathbf{A}^{(k)}_{l+1}, 
		%     \end{align}
	%     for $l=1,\dots, d-k+2$ for all $k=1,\cdots, d-1$, where $\mathbf{A}_{l}^{(k)}$ is the $l$th unfolding matrix of tensor $\mathcal{A}^{(k)}\in \mathbb{R}^{r_{k-1}\times I_k \times I_{k+1}\times \cdots \times I_d}$.
	% \end{proposition}

\begin{theorem}[TT-URV case]  \label{thm: 3.2}
	Given a tensor $\mathcal{A} \in \mathbb{R}^{I_1\times \cdots \times I_d}$. Assume that the rank-$r_k$ approximation error of its $k$th unfolding matrix $\mathbf{A}_k \in \mathbb{R}^{(I_1\cdots I_k)\times (I_{k+1}\cdots I_d)}$ from the URV decomposition is $\varepsilon_k$, that is,
	\begin{equation}\label{eq: URVcase}
		\mathbf{A}_k = \begin{bmatrix}
			\mathbf{U}^{(k)}_{1} & \mathbf{U}^{(k)}_{2}
		\end{bmatrix} \begin{bmatrix}
			\mathbf{R}^{(k)}_{11} & \mathbf{R}^{(k)}_{12} \\
			\mathbf{O} & \mathbf{R}^{(k)}_{22}
		\end{bmatrix} \begin{bmatrix}
			\mathbf{V}^{(k)\top}_{1} \\
			\mathbf{V}^{(k)\top}_{2}
		\end{bmatrix} = \mathbf{U}^{(k)}_{1}\mathbf{R}^{(k)}_{11}\mathbf{V}^{(k)\top}_{1} + \mathbf{E}^{(k)},
	\end{equation}
	where $\mathbf{R}^{(k)}_{11}\in \mathbb{R}^{r_k \times r_k}$, $\mathbf{E}^{(k)} = (\mathbf{U}^{(k)}_{1}\mathbf{R}^{(k)}_{12} +\mathbf{U}^{(k)}_{2}\mathbf{R}^{(k)}_{22})\mathbf{V}^{(k)\top}_{2}$, and $\| \mathbf{E}^{(k)} \| = \varepsilon_k$ for all $k=1,\dots, d-1$. Then the approximate tensor $\widehat{\mathcal{A}}$ with TT-ranks not higher than $\{r_k\}_0^d$, computed via the TT-URV algorithm~\ref{alg:TT-URV} in the right-to-left sweep, meets
	\[ \| \mathcal{A}-\widehat{\mathcal{A}} \|_F \le \sqrt{\sum_{k=1}^{d-1} \varepsilon_k^2}.\]
\end{theorem}

The TT-URV algorithm for prescribed accuracy can be obtained analogously by referring to Algorithms \ref{alg:TT-UTVver2} and \ref{alg:TT-URV}. It is not repeated here.

% \begin{algorithm}[!htpb]
	% 	\renewcommand{\algorithmicrequire}{\textbf{Input:}}
	% 	\renewcommand{\algorithmicensure}{\textbf{Output:}}
	% 	\caption{TT-URV algorithm for fixed TT-ranks (right-to-left sweep)}
	% 	\label{alg:TT-URV2}
	% 	\centering
	% 	\begin{algorithmic}[1]
		% 		\REQUIRE A $d$-dimensional tensor $\mathcal{A} \in \mathbb{R}^{I_1 \times \cdots \times I_d}$, fixed TT-ranks $\{ r_0,r_1, \dots, r_d \}$. 
		% 		\ENSURE Right-orthogonal TT-cores $\mathcal{G}^{(1)}, \dots, \mathcal{G}^{(d)}$ of the approximation $\widehat{\mathcal{A}}$ with TT-ranks $r_k$'s.
		
		% 		\STATE Temporary matrix: $\mathbf{C}=\mathrm{reshape}(\mathcal{A}, \left[ I_1\cdots I_{d-1}, I_d \right])$, i.e., the $(d-1)$th unfolding matrix $\mathbf{A}_{d-1}$.
		% 		\STATE for $k=d$ to $2$ do
		% 		\STATE \quad  $\mathbf{C}:=\operatorname{reshape}\left(\mathbf{C},\left[I_1\cdots I_{k-1}, I_{k}r_{k}\right]\right)$.
		% 		\STATE \quad Compute $\delta$-truncated ULV decomposition $\mathbf{C}=\widehat{\mathbf{U}} \widehat{\mathbf{R}} \widehat{\mathbf{V}}^{\top}+\mathbf{E}$ with $\|\mathbf{E}\|_F \leq \delta$ and $\widehat{\mathbf{R}}\in \mathbb{R}^{r_{k-1}\times r_{k-1}}$. 
		% 		\STATE \quad New core: $\mathcal{G}^{(k)}=\operatorname{reshape}\left(\widehat{\mathbf{V}}^{\top},\left[r_{k-1}, I_k, r_k\right]\right)$.
		% 		\STATE \quad  $\mathbf{C}= \widehat{\mathbf{U}} \widehat{\mathbf{R}}$.
		% 		\STATE end for
		% 		\STATE $\mathcal{G}^{(1)}=\mathbf{C}$.
		% 		%\STATE Return tensor $\widehat{\mathcal{A}}$ in TT-format with cores $\mathcal{G}_1, \dots, \mathcal{G}_d$.
		% 	\end{algorithmic}  
	% \end{algorithm}

Based on the above theoretical analysis, we recommend the TT-ULV algorithm in a left-to-right sweep to compute the left-orthogonal cores and the TT-URV algorithm in a right-to-left sweep to compute the right-orthogonal cores.

\begin{remark}
	In practice, both the ULV and URV decompositions can be flexibly adapted for computing the TT decomposition with left-orthogonal or right-orthogonal TT-cores. Taking TT-URV as an example, while our theoretical analysis indicates that it is better suited for a right-to-left update pattern, one can still use it to compute left-orthogonal TT cores for a given tensor. This is achieved by simply reversing the index order  of the input tensor and applying the TT-URV algorithm with a right-to-left sweep to update the computation on the reversed tensor. Finally, reverse the index order of the resulting TT-cores.
\end{remark}

\subsection{Computational cost}

The computational procedure for the SVD typically consists of two main stages: an initial reduction of the matrix to bidiagonal form, which is non-iterative, followed by an iterative QR phase. For an
$m\times n$ matrix with $m>n$, the approximate dominant term of the flop count for the SVD is $6mn^2$ \cite{golub2013matrix}. In contrast, the UTV decomposition imposes fewer structural constraints on middle factors than SVD, leading to lower computational overhead. Algorithms for UTV decomposition constitute a family of related methods, including the low-rank revealing UTV decompositions \cite{Fierro1997low}, the USV-plus decomposition \cite{Lee2018rankrev}, among others. The low-rank revealing  UTV algorithms have two modes: cold-start and warm-start. The warm-started algorithms assume that a good initial guess of a previous decomposition such as QR decomposition is available, they are designed to refine or update the decomposition efficiently via Givens rotations; The cold-started algorithms start from scratch without any initial guess, which builds the low-rank UTV decomposition directly from the original matrix iteratively. The USV-plus is also executed in an iterative manner, involving power or Lanczos iterations. Let $k$ denote the desired rank for the truncated decompositions. According to \cite[Table 2]{Fierro1997low}, the approximate dominant term of the flop count for cold-started UTV algorithms is $4mnkp+12mnk$, where $p$ denotes the average number of power or Lanczos iterations per deflation step, typically less than $5$. For the USV-plus decomposition, the complexity is dominated by $4mn(k+1)q$, with $q$ representing the average number of power iterations required for per singular value estimate. Based on these expressions, we derive the corresponding complexity estimates for the flop counts of the proposed TT decomposition algorithms.

In both the TT-SVD and TT-UTV algorithms, the predominant computational cost in each iteration arises from low-rank matrix approximation, implemented through either SVD or the more general UTV decompositions. For a $d$th-order tensor $\mathcal{A}\in \mathbb{R}^{I_1\times I_2\times \cdots \times I_d}$, the computational process needs to sequentially decompose $d-1$ matrices of size $I_1\times (I_2\cdots I_d)$, $r_1I_2\times (I_3\cdots I_d)$, $\ldots$, and $r_{d-2}I_{d-1}\times I_d$. To facilitate the comparison of computational complexity orders, we  assume that all dimensions of the input tensor are equal to $I$ and that all TT-ranks are taken as $r$. Then, the dominant term of the computational complexity for TT-SVD is
\[ 6(I^{d+1}+r^2I^d+r^2I^{d-1}+\cdots + r^2I^4 + rI^3). \]
For the TT-UTV algorithms based on cold-started low-rank UTV decompositions, the dominant term is
\[ (12+4p)(rI^d+r^2I^{d-1}+r^2I^{d-2}+\cdots +r^2I^2), \]
and for the TT-ULV algorithm based on the USV-plus decomposition, it is
\[4q(r+1)(I^d+rI^{d-1}+rI^{d-2}+\cdots +rI^2).\]
In comparison, the latter two terms are one order lower in the exponent of $I$ than that of TT-SVD. This computational advantage becomes particularly evident when $r$ is substantially smaller than 
$I$. In addition, there are other choices of UTV algorithms that are computationally efficient than SVD. For instance, the randUTV algorithm can make better use of fast BLAS routines and offers greater flexibility for parallel processing, leading to faster execution in practical applications \cite{Martinsson2019rand}.

% The computational savings of UTV over SVD are amplified in the TT decomposition of a $d$th-order tensor $\mathcal{A} \in \mathbb{R}^{I_1 \times I_2 \times \cdots \times I_d}$. Since the TT decomposition requires $d-1$ sequential matrix factorizations, the savings achieved by UTV decomposition at each iteration accumulate  across the entire hierarchy. As a result, the UTV-based methods enjoy a scaled computational advantage relative to the TT-SVD method, building upon their inherent efficiency at the single-matrix level.

%as the process necessitates $d-1$ sequential matrix factorizations. Consequently, the UTV-based methods achieve a multiplicatively scaled advantage from the matrix case over the TT-SVD method. 

%The computational savings offered by UTV over SVD for a single matrix decomposition are amplified in the TT decomposition of a $d$th-order tensor $\mathcal{A} \in \mathbb{R}^{I_1 \times I_2 \times \cdots \times I_d}$, since the TT framework necessitates $d-1$ sequential matrix factorizations.

\section{Experiments}
\label{sec: 4}
% In this section, we demonstrate the good performance of the TT-UTV algorithms for different applications. There are various available algorithms for rank-revealing UTV decomposition, such as the low-rank UTV algorithm \cite{Fierro1997low} packaged in the `UTV tools' \cite{fierro1999utv, fierro2005utv}, which will be used in the numerical tests. The experiments were conducted on a computer equipped with an Apple M4 chip (featuring a 10-core CPU) and 24 GB of RAM, using MATLAB R2025a. In our benchmark studies, the TT-SVD algorithm is executed using the economy-size SVD command svd(A, `econ') to ensure computational efficiency.

In this section, we demonstrate the favorable performance of the TT-UTV algorithms. The warm-started and cold-started versions for low-rank UTV decomposition \cite{Fierro1997low} are integrated into the MATLAB-based toolbox `UTVtools' \cite{fierro1999utv}. In this toolbox, the warm-started algorithms for ULV and URV decompositions are provided via the functions `lulv' and `lurv', respectively, while the cold-started versions are implemented through `lulv\_a' and `lurv\_a'. Another notable approach is the USV-plus decomposition \cite{Lee2018rankrev}, a rank-revealing ULV decomposition that iteratively estimates the singular values and singular vectors of a matrix one by one. The iteration terminates when the most recently estimated singular value falls below a prescribed threshold, thereby yielding the numerical rank and an approximate column space. To adapt the USV-plus algorithm for TT decomposition, we adopt its basic computational procedure, with a slight modification to the stopping criterion to accommodate both fixed-rank and fixed-precision TT-ULV algorithm.

The experiments were conducted on a computer equipped with an Apple M4 chip (featuring a 10-core CPU) and 24 GB of RAM, using MATLAB R2025a. As the benchmark, the TT-SVD algorithm is executed using the highly optimized built-in command svd(A, `econ') to ensure computational efficiency, that is, an economy-size SVD is computed for a matrix $\mathbf{A} \in \mathbb{R}^{m \times n}$, which computes $\mathbf{A} =\mathbf{U}\mathbf{\Sigma}\mathbf{V}^{\top} $ such that $\mathbf{U} \in \mathbb{R}^{m \times n}, \mathbf{\Sigma} \in \mathbb{R}^{n \times n}$, and $\mathbf{V} \in \mathbb{R}^{n \times n}$, if $m \geq n$. and the recently proposed randomized TT-SVD (RandTT-SVD) algorithm \cite{che2026efficient} is also included for comparison.  Note that the TT-SVD and TT-ULV algorithms are executed via a left-to-right sweep. In contrast, the TT-URV algorithms are applied to the tensors with their indices reversed and proceed via a right-to-left sweep, in this case, both the ordering and indices of the resultant cores and the sequence of TT ranks undergo a corresponding reversal. All of them yield left-orthogonal TT-cores.

To benchmark the proposed algorithms, we evaluated their efficiency and accuracy on multiple tensors under both fixed-rank and given-precision conditions. For the tensor data $\mathcal{A}$, let $\widehat{\mathcal{A}}$ denote the tensor reconstructed from the TT formats generated by these algorithms. The reconstruction accuracy is measured by the Relative Squared Error (RSE), defined as
\[ \mathrm{RSE} = \frac{\| \widehat{\mathcal{A}}-\mathcal{A}\|_F}{\| \mathcal{A} \|_F}. \]
The computational efficiency of each algorithm is benchmarked by its runtime, measured in seconds. To mitigate the effects of randomness, the reported runtime is the average elapsed time over 10 independent executions; for randomized algorithms, the RSE and estimated TT ranks are also averaged.

% In our benchmark studies, we evaluate both deterministic and randomized approaches. The TT-SVD algorithm is executed based on the highly optimized economy-size SVD command svd(A, 'econ'), while the recently proposed randomized TT-SVD algorithm is also tested as a counterpart.

\subsection{Synthetic data}

\textbf{Example 1} Similar to \cite{sun2020low}, we generate tensors $\mathcal{A}\in \mathbb{R}^{160\times 160\times 160\times 160}$ as
\[  \mathcal{A} =  \mathcal{P} + \beta \mathcal{N} \]
with a clean signal tensor $\mathcal{P}=\mathcal{G}_1 \times{ }_2^1 \mathcal{G}_2 \times{ }_3^1 \mathcal{G}_3 \times{ }_3^1 \mathcal{G}_4$, where the entries of $\mathcal{G}_1 \in \mathbb{R}^{160 \times 20}, \mathcal{G}_2, \mathcal{G}_3 \in \mathbb{R}^{20 \times 160 \times 20}$ and $\mathcal{G}_4 \in \mathbb{R}^{20 \times 160}$ are independent Gaussian variables with zero mean and unit variance, and the entries of the noise tensor  $\mathcal{N} \in \mathbb{R}^{160\times 160\times 160\times 160}$ form a family of independent  standard normal random variables. Different noise magnitudes $\beta$ are added to the tensor $\mathcal{P}$, so that the resulting tensors $\mathcal{A}$ exhibit various signal-to-noise ratio (SNR) levels, defined as
\[ \mathrm{SNR}[\mathrm{~dB}]=10 \log \left(\frac{\|\mathcal{P}\|_F^2}{\|\beta \mathcal{N}\|_F^2}\right). \]

For tensors $\mathcal{A} \in \mathbb{R}^{160 \times 160 \times 160 \times 160}$ with different SNR levels, we fix the TT-rank to $\{20, 20, 20\}$ and employ various algorithms, including TT‑SVD, RandTT‑SVD, and the proposed TT‑UTV algorithms to compute the TT approximations. The inner subroutines of TT‑UTV utilize warm‑started ULV, URV, and USV‑plus algorithms. The intermediate steps of the cold-start low-rank UTV algorithm generate excessively large arrays, and its application to large-scale tensor computations leads to out-of-memory errors on our devices. Thus, we do not report the performance of this algorithm in the following sections. Following the standard implementations in the UTV toolbox by Fierro and Hansen, the number of Lanczos iterations $p$ in each refinement is set not higher than $5$. The number of power iteration $q$ of USV-plus is totally determined by the computational accuracy.

As can be seen from Fig. \ref{fig: ARSEs}, the TT approximations obtained by these algorithms achieve very similar accuracy, with RandTT‑SVD exhibiting slight fluctuations. For runtime, RandTT‑SVD is the fastest, followed by TT‑UTV (USV‑plus). The TT‑UTV algorithms based on low‑rank UTV decomposition is the slowest. It is worth noting that, although TT‑UTV (lulv, lurv) have lower theoretical computational complexity than TT‑SVD, their practical runtime are longer, possibly because the implementation does not leverage higher-level linear algebra subroutines, e.g., BLAS3.

For the tensor $\mathcal{A}$ with $\mathrm{SNR}=30$, we test the performance of the algorithms under various TT‑ranks. As illustrated in Fig. \ref{fig: Aranks}, TT‑SVD achieves optimal accuracy, while TT‑UTV yields a comparable level of accuracy. It is worth noting that, in terms of runtime, TT‑ULV (USV‑plus) exhibits a clear advantage over TT‑SVD, and the advantage becomes more pronounced as the TT‑rank decreases, which is precisely the manifestation of the theoretical complexity analysis. RandTT‑SVD offers the fastest speed, but at the cost of a noticeable loss in accuracy.
\begin{figure}[H]
	\centering
	\includegraphics[width = 0.96\textwidth]{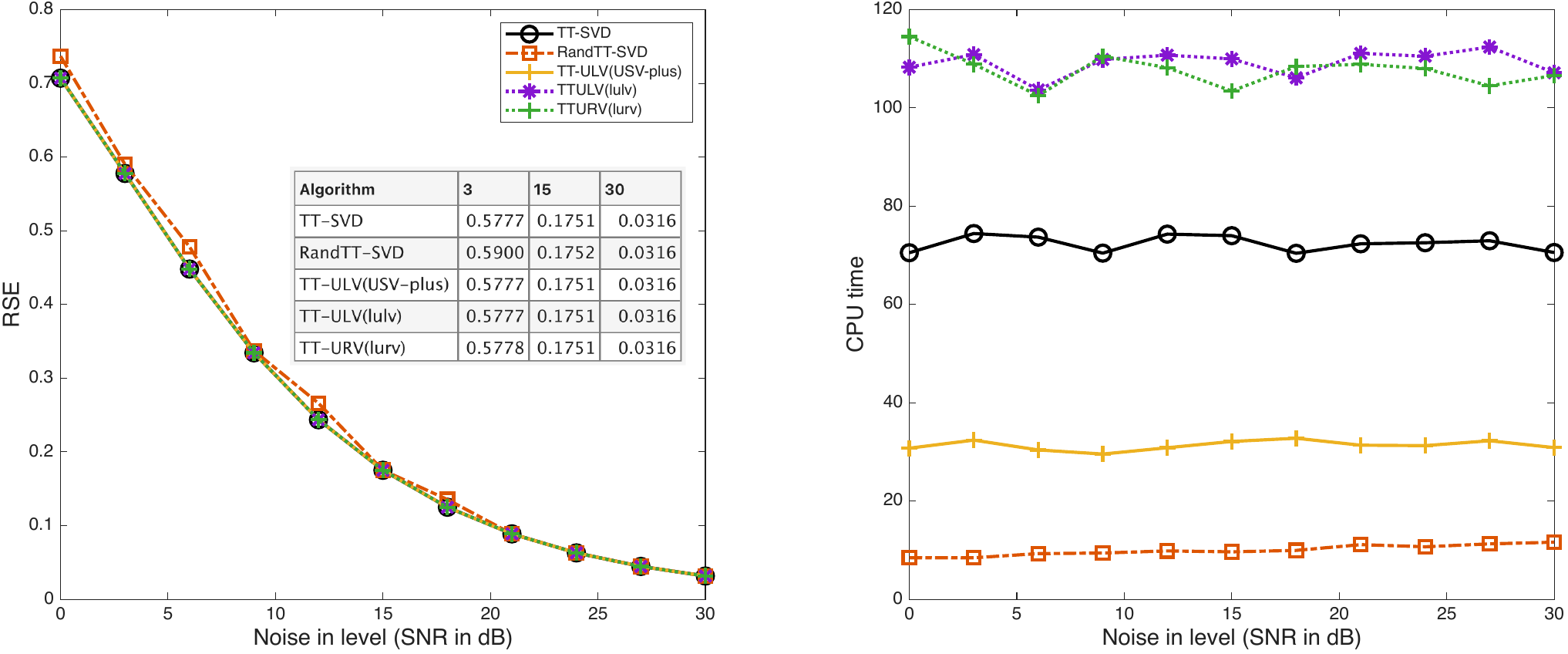}
	\caption{For different tensors $\mathcal{A}$ with $\mathrm{SNR}=0,5,\ldots, 25, 30$, and given fixed TT-rank $\{20,20,20 \}$, the results by applying TT-SVD, RandTT-SVD, and the TT-UTV with various UTV algorithms.}
	\label{fig: ARSEs}
\end{figure}
\begin{figure}[htbp!]
	\centering
	\includegraphics[width = 0.96\textwidth]{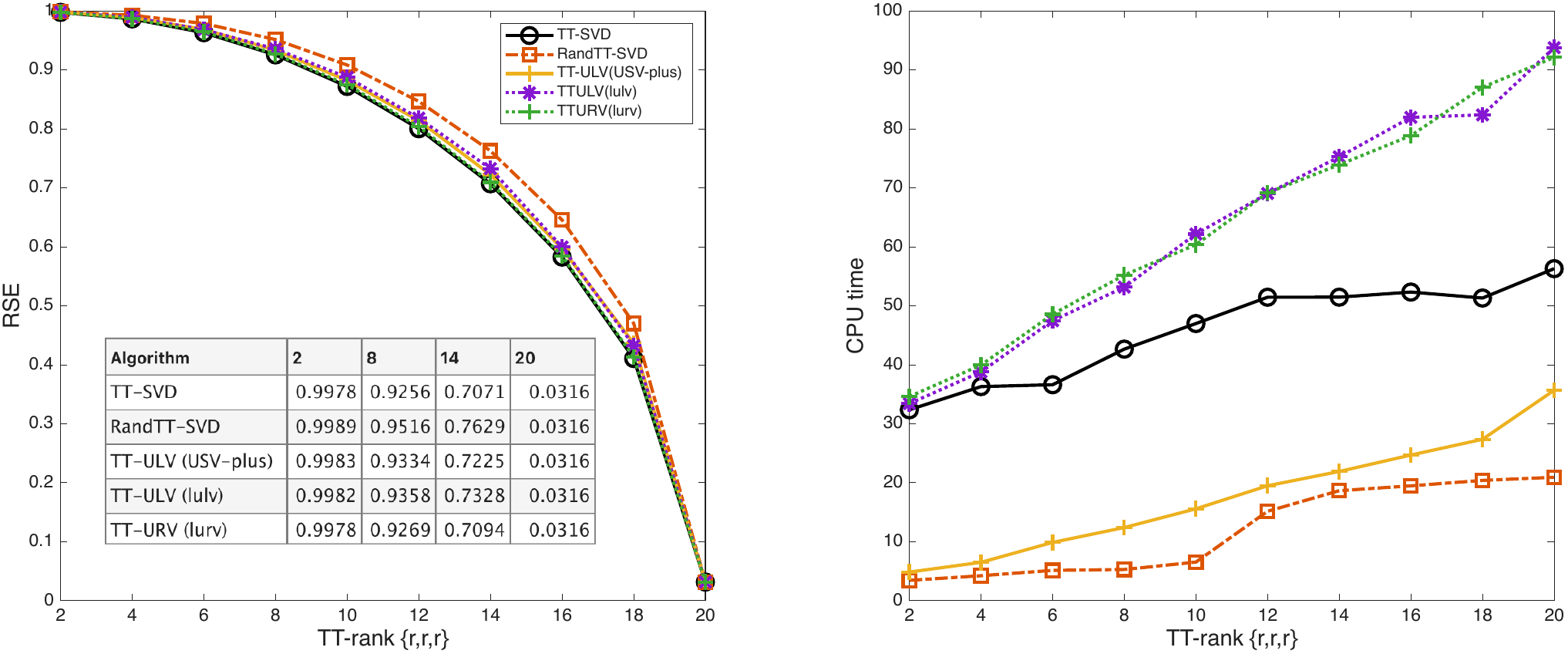}
	\caption{For given tensor $\mathcal{A}$ with $\mathrm{SNR}=30$, and different fixed TT-rank $\{r,r,r\}$ with $r=2,4,\ldots, 18,20$, the results by applying TT-SVD, RandTT-SVD, and the TT-UTV with various UTV algorithms.}
	\label{fig: Aranks}
\end{figure}

We next test the TT‑ALS algorithm \cite{Holtz2012the}, which belongs to the class of optimization methods. Given an initial guess, it iteratively updates the TT cores by solving alternating least‑squares subproblems. Here, we randomly set the initial values, and to mitigate the influence of randomness, we report the averaged results over ten independent runs. For a direct visual comparison among the algorithms, we plot the convergence curves of TT‑ALS in terms of CPU time, where other algorithms are represented as individual points; the closer a point lies to the bottom‑left corner, the better the algorithmic performance. The results in Fig. \ref{fig: TT-ALS} show that, for the two fixed TT‑rank experiments, TT‑ALS eventually achieves excellent accuracy but requires a long CPU time.
\begin{figure}[htbp!]
	\centering
	\includegraphics[width = 0.96\textwidth]{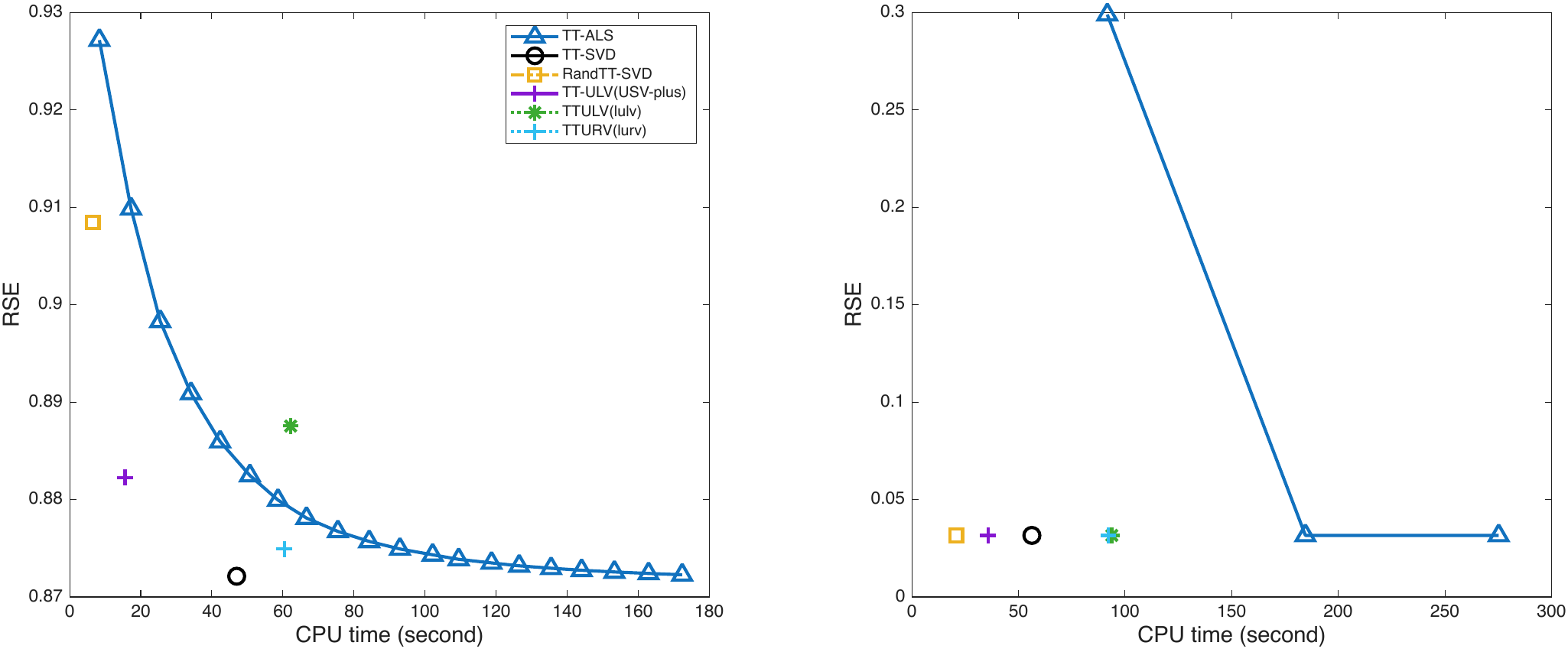}
	\caption{For given tensor $\mathcal{A}$ with $\mathrm{SNR}=30$, and the fixed TT-ranks $\{10,10,10\}$  and $\{20,20,20\}$, the results by applying TT-ALS, TT-SVD, RandTT-SVD, and the TT-UTV with various UTV algorithms.}
	\label{fig: TT-ALS}
\end{figure}

\textbf{Example 2} To test the fixed-precision versions of our algorithms, the following tensor is obtained by discretizing a smooth function as follows:

$$
\begin{aligned}
\mathcal{B}\left(i_1, i_2, i_3, i_4\right) & =\sin \sqrt{\sum_{k=1}^4\left(\frac{i_k-1}{I-1}\right)^2}, 
\end{aligned}
$$
with $i_1, i_2, i_3, i_4=1,2, \ldots, 160$. The type of tensor $\mathcal{B}$ comes from \cite{lestandi2021numerical}.

\begin{table}[htbp]
\centering
\caption{For given different error tolerance $\epsilon > 0$, the results by applying
TT-SVD, RandTT-SVD, and TT-UTV algorithms to estimate the TT-rank of the tensor $\mathcal{B}$.}
\label{tab: Beps}
\begin{adjustbox}{width=1\linewidth}
\begin{tabular}{c c c c c c}
\toprule % 使用booktabs的顶部线，如不需任何线可改为% \toprule 注释掉
$\varepsilon$ & TT-SVD & RandTT-SVD & TT-ULV(USV-plus) & TT-ULV(lulv) & TT-URV(lurv) \\
\midrule % 如需线可保留，否则注释
$10^{-2}$ & (2,2,2) & (3,3,3) & (2,2,2) & (2,2,2) & (2,2,2) \\
$10^{-3}$ & (3,3,3) & (5,4,4) & (3,3,3) & (3,3,3) & (3,3,3) \\
$10^{-4}$ & (5,5,5) & (6,6,6) & (5,5,5) & (5,5,5) & (5,5,5) \\
$10^{-5}$ & (7,7,6) & (8,9,8) & (6,7,6) & (7,7,6) & (6,7,7) \\
\bottomrule % 如需线可保留，否则注释
\end{tabular}
\end{adjustbox}
\end{table}

For the tensor $\mathcal{B}$, given different prescribed error tolerances $\epsilon>0$, the resulting TT‑ranks obtained by various algorithms are reported in Table \ref{tab: Beps}. The TT-ranks obtained by the several TT‑UTV variants are highly consistent with those from TT‑SVD, while RandTT‑SVD yields slightly higher TT‑ranks, implying that the post‑processing of its computed TT cores requires larger memory and computational effort. The corresponding RSEs and computation times are displayed in Fig. \ref{fig: Beps}. The left panel indicates that the TT approximation errors produced by all algorithms are strictly less than the given tolerances, and the theoretical error bound is tight. The right panel shows that the TT‑ULV (USV‑plus) algorithm not only attains very low TT‑ranks but also exhibits a remarkable enhancement in computational efficiency.

\begin{figure}[htbp!]
	\centering
	\includegraphics[width = 0.96\textwidth]{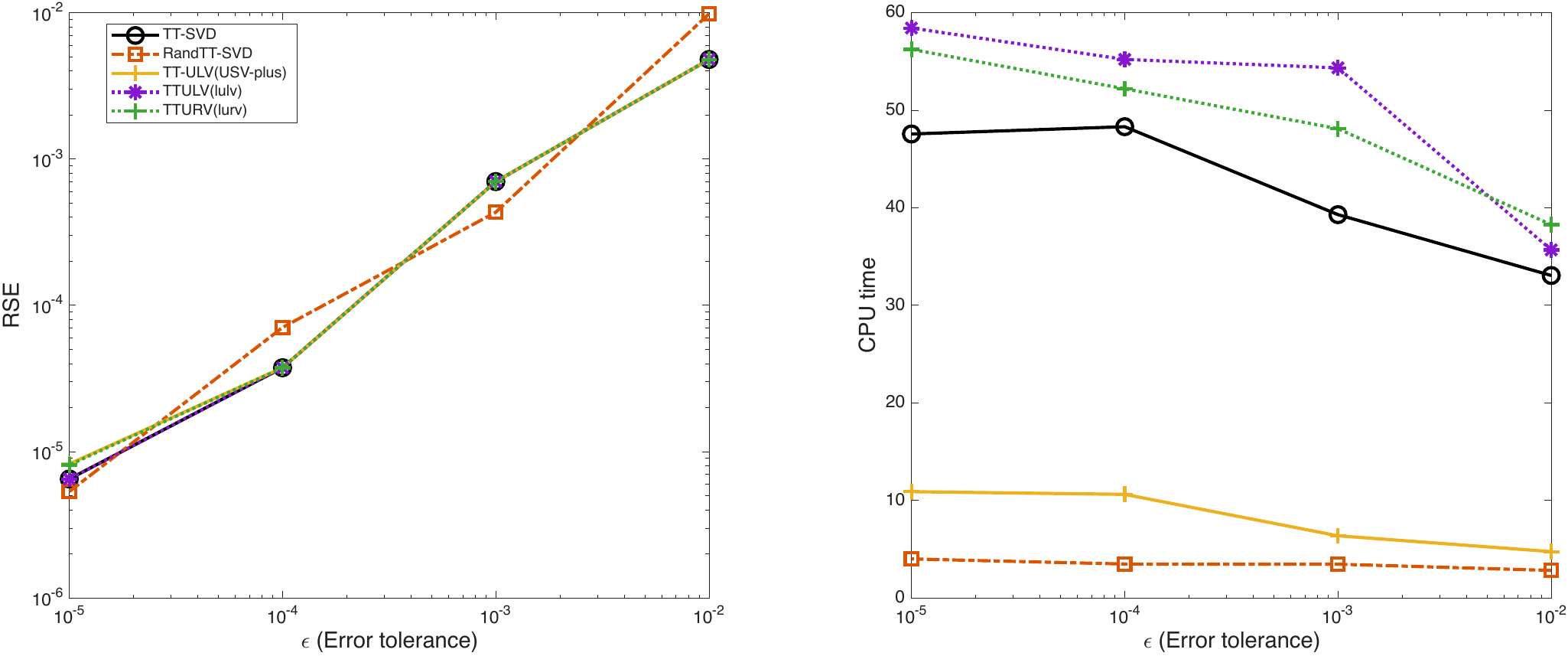}
	\caption{For given different error tolerance $\epsilon > 0$, the results by applying
TT-SVD, RandTT-SVD, and TT-UTV algorithms to the tensor $\mathcal{B}$.}
	\label{fig: Beps}
\end{figure}

\subsection{Real datasets}

In the sequel, we test these algorithms on real datasets, including video data compression and MRI data recovery.

\subsubsection{Video data}

The experimental dataset was derived from a source video\footnote{$https://youtu.be/X9cCgErwqQg$} comprising 3,757 frames at a resolution of $2160\times 3840$. Constrained by the substantial memory requirements of the raw video, each frame was downsampled to $72\times 128$ pixels using the MATLAB `imresize' command \cite{ko2020fast}, which performs a two‑dimensional image resizing operation using bicubic interpolation by default. The subsequent separation of the RGB channels produced a tensor of size $72\times 128 \times 3$ for each frame. Hence,  a fourth-order tensor of dimensions $3757\times 72\times 128\times 3$ is formed, where the first dimension corresponds to the batch index.

The performance of these algorithms was evaluated under fixed TT-ranks $(r_1, r_2, r_3)$, as summarized in Table~\ref{table: 1}. Numerical results demonstrate that the TT-UTV (USV-plus) algorithm achieves effective reductions in computational time while preserving accuracy comparable to that of TT-SVD. This advantage is particularly pronounced for low-rank cases, which can achieve speedups of more than ten times in some instances. In this set of experiments, the TT‑UTV algorithms based on low‑rank UTV decomposition obtain a level of accuracy that is fully comparable to that of TT‑SVD, albeit with longer actual computation times. The RandTT-SVD algorithm offers the fastest runtime at the cost of a non-negligible loss in accuracy.

\begin{table}[htbp!]
	\centering
	\caption{For different TT-rank, the numerical results by applying TT-SVD, RandTT-SVD, and TT-UTV algorithms on the video data tensor.}
	\label{table: 1}
	\begin{adjustbox}{width=1\linewidth}
   \begin{tabular}{c!{\thickvline}cccc!{\thickvline}cccc}
		\thickhline & \multicolumn{4}{c!{\thickvline}}{ Time (s) } & \multicolumn{4}{c}{ RSE } \\
		\mediumhline  $(r_1, r_2)$ & $(5,5,2)$ & $(10,10,2)$ & $(15,15,2)$ & $(20,20,2)$ & $(5,5,2)$ & $(10,10,2)$ & $(15,15,2)$ & $(20,20,2)$ \\
		\hline TT-SVD & 18.4 & 18.8 & 19.0 & 19.1 & 0.1202 & 0.0955 & 0.0849 & 0.0794 \\
        \hline RandTT-SVD & 0.3 & 0.4 & 0.5 & 0.5 & 0.1245 & 0.0985 & 0.0883 & 0.0830 \\
        \hline TT-ULV (USV-plus) & 1.0 & 1.9 & 2.7 & 3.8 & 0.1205 & 0.0956 & 0.0851 & 0.0797 \\
		\hline TT-ULV (lulv) & 22.8 & 24.5 & 25.6 & 27.6 & 0.1202 & 0.0961 & 0.0850 & 0.0796 \\
        % \hline {\color{blue}TT-ULV (lulv\_a)} & 8.6 & 16.6 & 24.3 & 32.6 & 0.1179 & 0.0938 & 0.0832 & 0.0777 \\
		\hline TT-URV (lurv) & 39.2 & 51.5 & 58.3 & 62.8 & 0.1202 & 0.0960 & 0.0850 & 0.0796 \\
        % \hline {\color{blue}TT-URV (lurv\_a)} & 9.2 & 16.6 & 24.6 & 32.9 & 0.1179 & 0.0938 & 0.0832 & 0.0777 \\
		\thickhline
	\end{tabular}
\end{adjustbox}
\end{table}

We also evaluated the fixed-precision variants of the algorithms under different error tolerance settings, with the error weights $w_k$ set uniformly to $1/\sqrt{3}$. As presented in Table~\ref{table: 2},  the TT-ranks determined by the TT-UTV algorithms closely match those produced by TT-SVD. As expected, larger error tolerances lead to smaller TT-ranks. In all cases, the actual computational errors remain strictly below the specified error tolerances $\varepsilon$, which also confirms the effectiveness of our theoretical error bounds derived for the TT-UTV algorithms. In addition, the TT-UTV algorithm based on USV-plus exhibits a great advantage in computational efficiency. Although the RandTT-SVD achieves the smallest runtime, it produces substantially larger TT-ranks than other methods. Consistent with complexity analysis, the computational efficiency of the UTV-based algorithms improves as the resultant ranks become smaller.
\begin{table}[htbp!]
	\centering
	\caption{For different prescribed tolerance $\epsilon>0$, the numerical results by TT-SVD, RandTT-SVD, and TT-UTV algorithms on the video data tensor.}
	\label{table: 2}
	\begin{adjustbox}{width=\linewidth}
    \begin{tabular}{c!{\thickvline}ccc!{\thickvline}ccc!{\thickvline}ccc}
		\thickhline  & \multicolumn{3}{c!\thickvline}{$\varepsilon=0.09$} & \multicolumn{3}{c!\thickvline}{$\varepsilon=0.1$} & \multicolumn{3}{c}{$\varepsilon=0.11$} \\
		 \hline Algorithm & $(r_1, r_2, r_3)$ & Time (s) & RSE & $(r_1, r_2, r_3)$ & Time (s) & RSE & $(r_1, r_2, r_3)$ & Time (s) & RSE \\
		\mediumhline TT-SVD & $(51,16,2)$ & 24.1 & 0.0758 & $(40,11,2)$ & 23.8 & 0.0883 & $(31,8,2)$ & 24.4 & 0.0889 \\
        \hline RandTT-SVD & $(129,41,2)$ & 3.0 & 0.0759 & $(99,28,2)$ & 2.4 & 0.0910 & $(82,24,3)$ & 2.1 & 0.0885 \\
        \hline TT-ULV (USV-plus) & $(55,17,2)$ & 10.3 & 0.0742 & $(43,12,2)$ & 8.0 & 0.0812 & $(33,8,2)$ & 5.5 & 0.0883 \\
		\hline TT-ULV (lulv) & $(52,16,2)$ & 35.5 & 0.0760 & $(41,11,2)$ & 33.6 & 0.0834 & $(32,8,2)$ & 31.5 & 0.0889 \\
		% \hline {\color{blue}TT-ULV (lulv\_a)} & $(38,6)$ & 64.0 & 0.0882 & $(27,4)$ & 46.6 & 0.0996 & $(19,3)$ & 32.5 & 0.1063 \\
        \hline TT-URV (lurv) & $(51,16,2)$ & 36.2 & 0.0760 & $(40,11,2)$ & 33.0 & 0.0834 & $(31,8,2)$ & 30.5 & 0.0890 \\
        % \hline {\color{blue}TT-URV (lurv\_a)} & $(38,6)$ & 62.7 & 0.0882 & $(27,4)$ & 49.9 & 0.0995 & $(19,3)$ & 32.6 & 0.1063 \\
		\thickhline
	\end{tabular}
	\end{adjustbox}
\end{table}

\subsubsection{Image compression}

In this experiment, we test the performance of TT-UTV algorithms on image compression. Three color images are picked from the Berkeley Segmentation dataset \cite{martin2001database}, which are of size $321\times 481$ pixels. In order to facilitate the dimension factorizations, they are resized into $324\times 486$ dimensions using bicubic interpolation via the MATLAB `imresize' command \cite{ko2020fast}. Then, each processed image is represented by three $324 \times 486$ matrices corresponding to its RGB channels. We first reshape the channel matrices into higher-order tensors of size $18\times 18\times 18\times 27$, and then compute their low-rank approximation using these TT decomposition algorithms. 

\begin{figure}[htbp!]
	\centering
	\subfigure{\includegraphics[width = 0.32\textwidth]{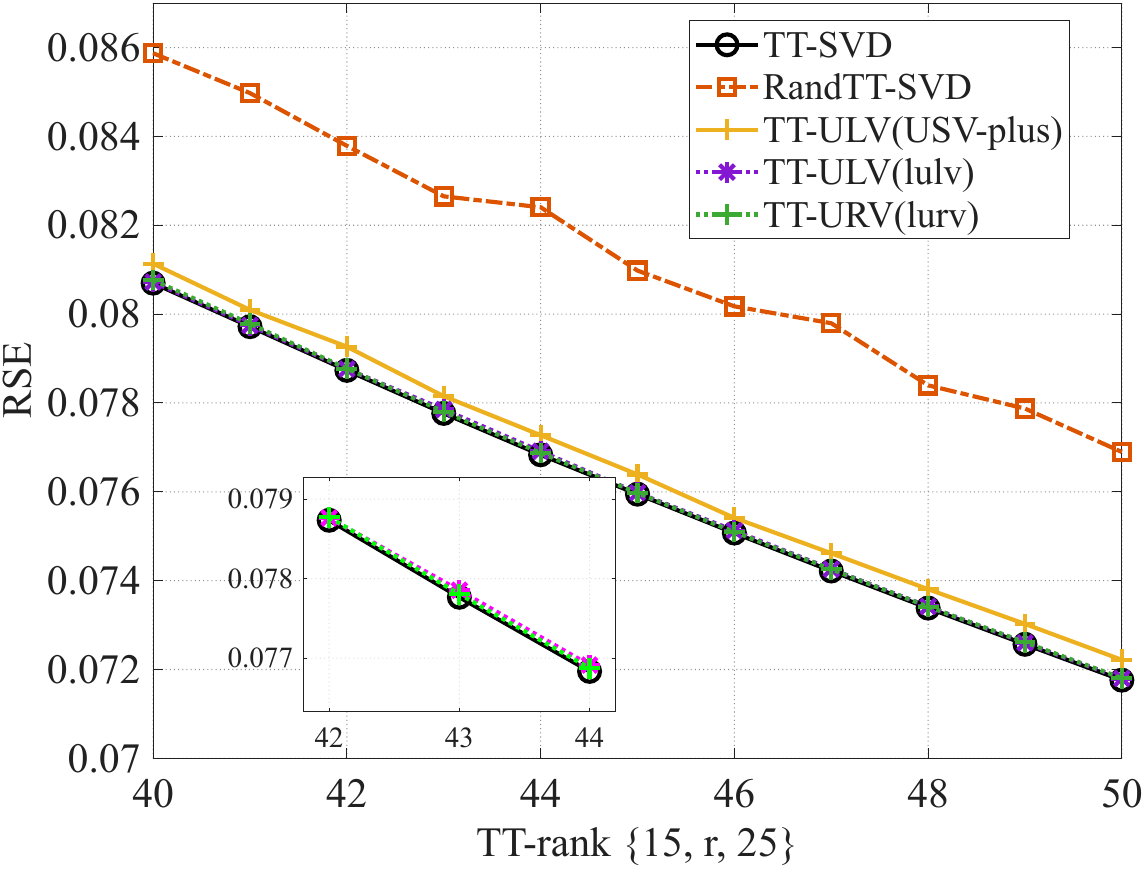}}
	\subfigure{\includegraphics[width = 0.32\textwidth]{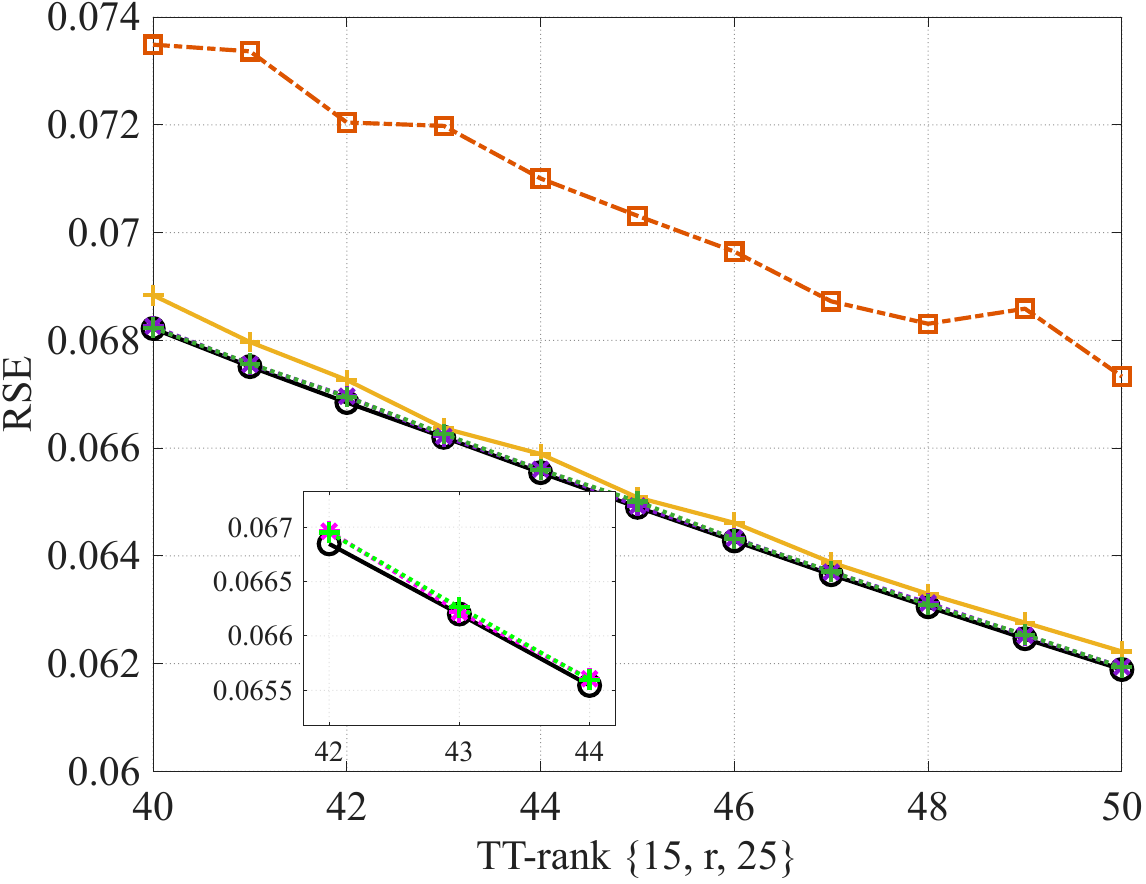}}
	\subfigure{\includegraphics[width = 0.32\textwidth]{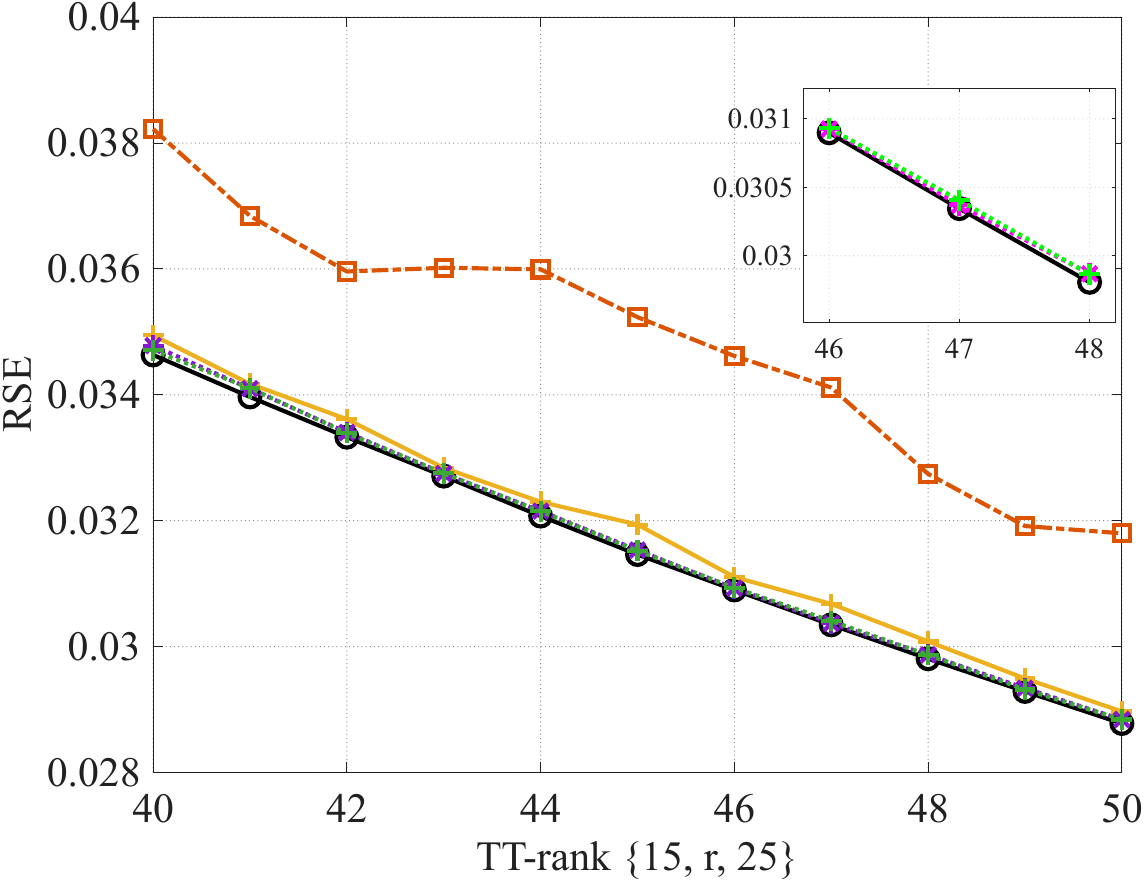}}
	\caption{For given TT-ranks $\{ 15,r,25 \}$ with $r=40,41,\ldots, 50$, the RSEs of TT compressions by applying TT-SVD, RandTT-SVD, and TT-UTV algorithms on these images.}
	\label{fig: RSEsimage}
\end{figure}
\begin{figure}[htbp!]
	\centering
	\captionsetup[subfigure]{labelformat=empty}
	\subfigure{\includegraphics[width = 0.24\textwidth]{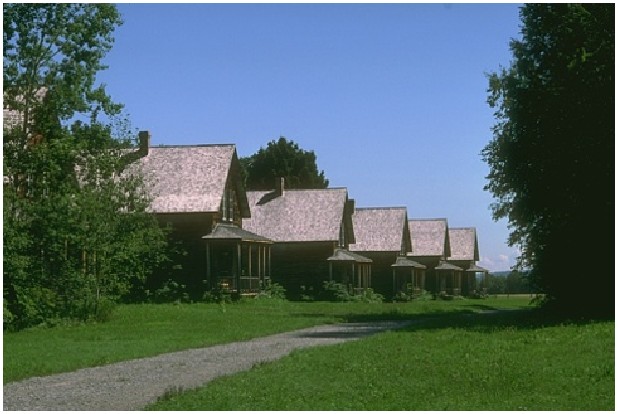}}
	\subfigure{\includegraphics[width = 0.24\textwidth]{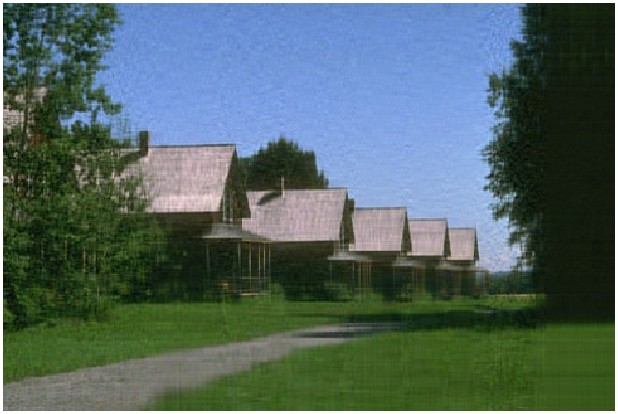}}
	\subfigure{\includegraphics[width = 0.24\textwidth]{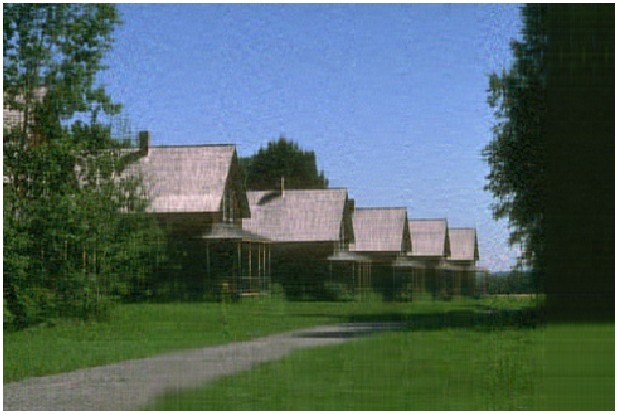}}
	\subfigure{\includegraphics[width = 0.24\textwidth]{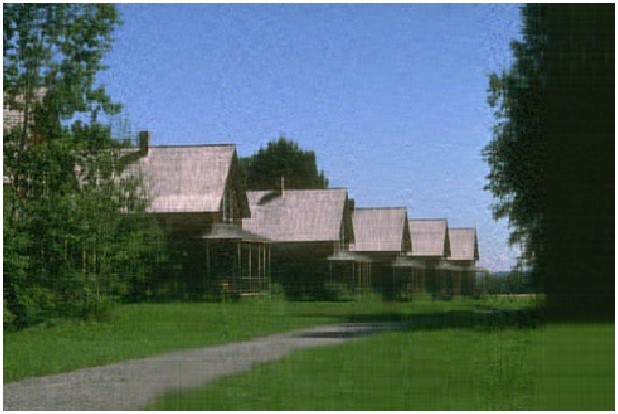}} \\
	\subfigure{\includegraphics[width = 0.24\textwidth]{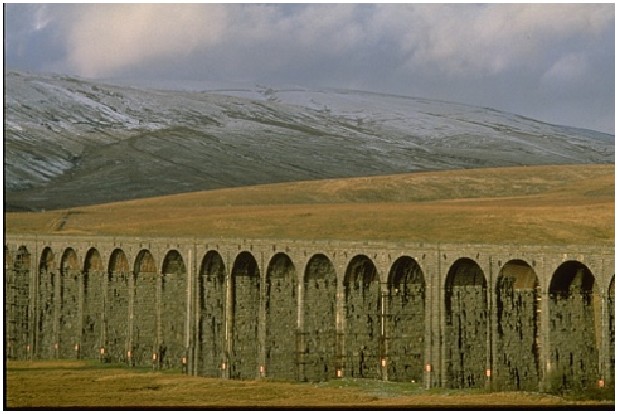}}
	\subfigure{\includegraphics[width = 0.24\textwidth]{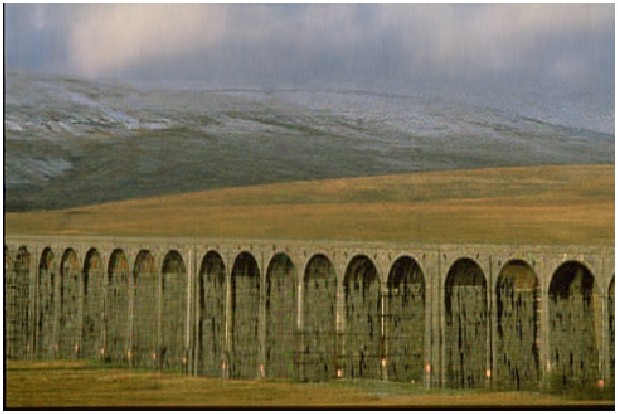}}
	\subfigure{\includegraphics[width = 0.24\textwidth]{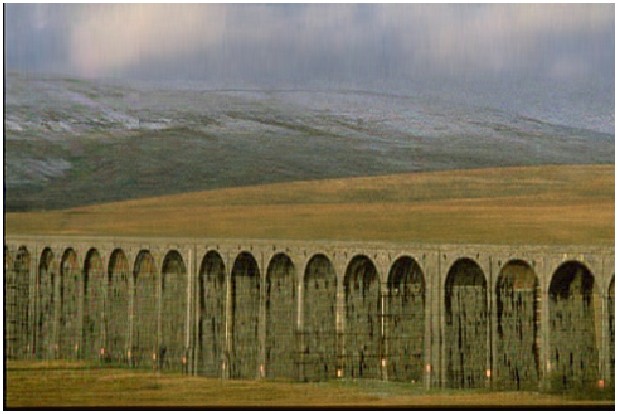}}
	\subfigure{\includegraphics[width = 0.24\textwidth]{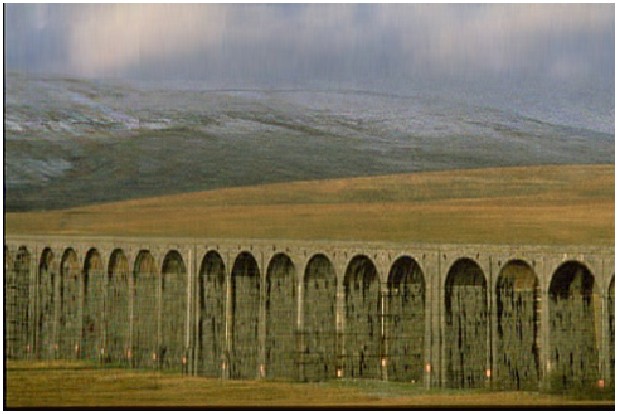}} \\
	\captionsetup[subfigure]{labelformat=simple}
	\setcounter{subfigure}{0}
	\subfigure[Original images]{\includegraphics[width = 0.24\textwidth]{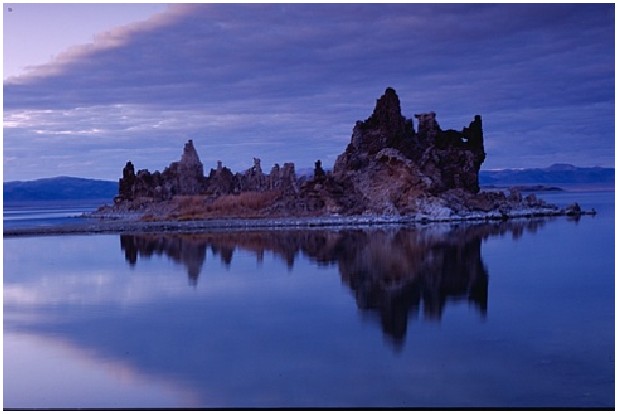}}
	\subfigure[TT-SVD]{\includegraphics[width = 0.24\textwidth]{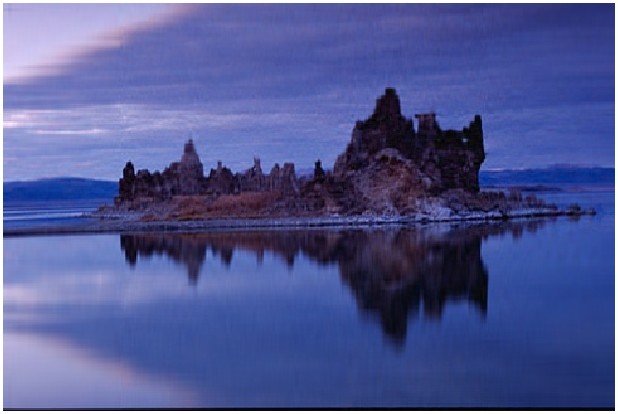}}
	\subfigure[TT-ULV]{\includegraphics[width = 0.24\textwidth]{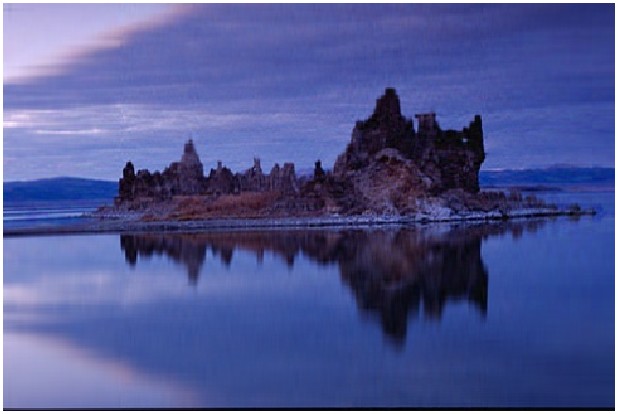}}
	\subfigure[TT-URV]{\includegraphics[width = 0.24\textwidth]{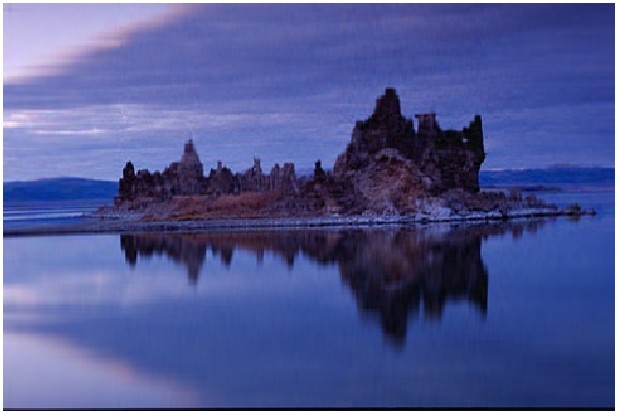}}
	\caption{Image compression via TT format. From left to right: original images, compressed versions from the TT approximations obtained via TT-SVD, TT-ULV (lulv), and TT-URV (lurv), respectively.}
	\label{fig:images}
\end{figure}

Fig.~\ref{fig: RSEsimage} reports the RSEs under varying TT-ranks $(1,15,r,25,1)$, where $r$ ranges from 40 to 50 by increments of 1. The results indicate that the RSEs achieved by the TT-SVD and TT-UTV algorithms are close across all tested ranks. Fig.~\ref{fig:images} demonstrates the compression performance on three test images under a fixed TT-rank of $(1, 15, 45, 25, 1)$. Visually, these compressed images are too similar to distinguish their resolutions and retain most of the structural and textural information from the original images.

\subsubsection{MRI data completion}

When a portion of the information in high-dimensional data is missing or distorted, we can recover the data based on TT-format tensor completion, which can be solved via the Riemannian gradient descent (RGrad) method \cite{Steinlechner2016}. Unlike the traditional gradient descent method in Euclidean spaces, the RGrad method uses the gradients on Riemannian manifolds formed by the fixed-rank TT-format tensors. From an initial TT-format tensor in the Riemannian manifold, the first step in each iteration is computing the Riemannian gradient $\mathcal{RG}(\mathcal{X}_k)$ at the current point $\mathcal{X}_k$. Then compute the estimated tensor $\widehat{\mathcal{X}}_{k+1} = \mathcal{X}_k - \alpha_k \mathcal{RG}(\mathcal{X}_k)$, where $\alpha_k$ is the step size. Actually $\widehat{\mathcal{X}}_{k+1}$ may not locate in the prescribed Riemannian manifold, hence, the next step retracts $\widehat{\mathcal{X}}_{k+1}$ back to the manifold by TT-SVD algorithm to generate the next iteration point $\mathcal{X}_{k+1}$ \cite{Steinlechner2016, Cai2022provable}. In this experiment, we explore the effectiveness of the RGrad method while the retraction step in each iteration is substituted by the TT-UTV algorithms. The step size $\alpha_k $ is determined via backtracking line search, computed as $\alpha_k = \frac{\left \langle \mathrm{P}_{\Omega}\mathcal{RG}(\mathcal{X}_k), \mathrm{P}_{\Omega}\mathcal{X}- \mathrm{P}_{\Omega}\mathcal{X}_k  \right \rangle }{ \left \langle \mathrm{P}_{\Omega}\mathcal{RG}(\mathcal{X}_k),  \mathrm{P}_{\Omega}\mathcal{RG}(\mathcal{X}_k) \right \rangle }$, where $\mathcal{X}$ denotes the original tensor, $\mathrm{P}_{\Omega}$ is the projection operator onto the index set $\Omega$ of observed entries. The initial tensor is uniformly set by the sequential spectral initialization \cite{Cai2022provable}.

The experimental results are evaluated by RSE, peak signal-to-noise Ratio (PSNR) \cite{ko2020fast}, and structural similarity index (SSIM) \cite{wang2004image}, which are commonly used measures in data inpainting tasks. Let $\mathcal{M}, \widehat{\mathcal{M}} \in \mathbb{R}^{I_1\times \cdots \times I_d}$ denote the ground-truth tensor and the recovered tensor, respectively. Then, the PSNR metric is defined as
\begin{align*} 
	%\mathrm{RSE} &= \frac{\| \widehat{\mathcal{M}}-\mathcal{M}\|_F}{\| \mathcal{M} \|_F}, \\
	\mathrm{PSNR} &= 10\log_{10}\frac{\mathcal{M}_{max}^2}{\frac{1}{\prod_{k=1}^{d}I_k}\| \widehat{\mathcal{M}} - \mathcal{M} \|_F^2},
\end{align*}
where $\mathcal{M}_{max}$ is the maximum value in the ground-truth tensor $\mathcal{M}$. {The SSIM for the MRI data was computed on a slice-by-slice basis for the brain images, and the values were subsequently averaged to yield a final score. Smaller RSE, larger PSNR, and larger SSIM indicate better recovery performance.

We use the cubical magnetic resonance imaging  (MRI) data\footnote{$https://brainweb.bic.mni.mcgill.ca/brainweb/selection\_normal.html$}, which is a third-order tensor denoted by $\mathcal{M}$ of size $181\times 217\times 181$. Firstly, we determine the truncated TT-ranks of this tensor for the Riemannian manifold. We set the prescribed RSE $\varepsilon = 0.1$ for the TT decomposition via the TT-SVD algorithm, which produces the approximate TT tensors with TT-ranks equal to $(1,55,41,1)$. We choose a slice $\mathcal{M}_{::36}$ from the MRI data shown in Fig.~\ref{fig:RGrad}. The four subfigures in the first line present the original slice and the compressed ones via various TT decomposition algorithms. They will serve as references for the same slice after tensor completion.
\begin{figure}[htbp!]
	\centering
	\subfigure[Original slice]{\includegraphics[width = 0.24\textwidth]{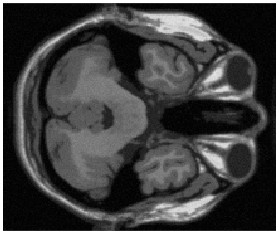}}
	\subfigure[TT-SVD compression]{\includegraphics[width = 0.24\textwidth]{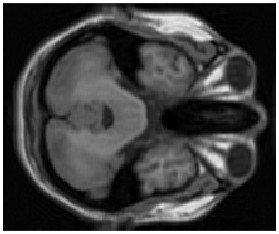}} 
	\subfigure[TT-ULV compression]{\includegraphics[width = 0.24\textwidth]{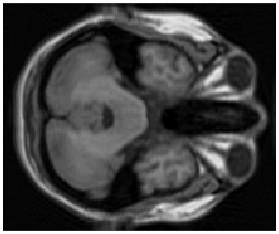}} 
	\subfigure[TT-URV compression]{\includegraphics[width = 0.24\textwidth]{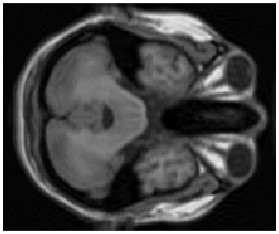}}\\
	\subfigure[Original slice with missing pixels]{\includegraphics[width = 0.24\textwidth]{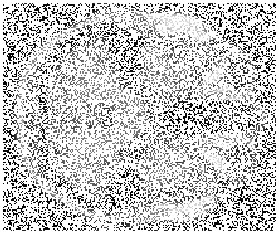}} 
	\subfigure[RGrad method with TT-SVD]{\includegraphics[width = 0.24\textwidth]{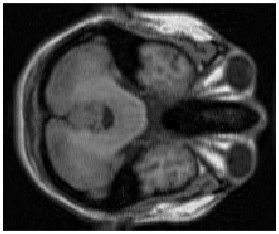}} 
	\subfigure[RGrad method with TT-ULV (lulv)]{\includegraphics[width = 0.24\textwidth]{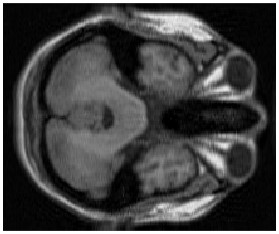}} 
	\subfigure[RGrad method with TT-URV (lurv)]{\includegraphics[width = 0.24\textwidth]{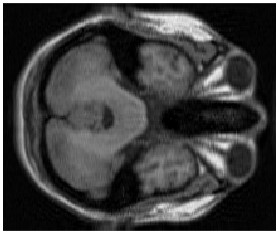}}
	\caption{The reconstruction performance on the brain slice $\mathcal{M}_{::36}$.}
	\label{fig:RGrad}
\end{figure}

Afterwards, to test the effectiveness of the RGrad method with TT-UTV algorithms, 70\% elements of the original MRI data $\mathcal{M}\in \mathbb{R}^{181\times 217 \times 181}$ are randomly removed, and we would like to recover them. Three experiments are performed using the RGrad method under fixed TT-ranks $(1,55,41,1)$, where the retracting steps are computed by TT-SVD, TT-ULV (USX-plus), TT-ULV (lulv), and TT-URV (lurv) algorithms, respectively. The second row in Fig.~\ref{fig:RGrad} shows the original slice with missing pixels and the recovered slices by RGrad methods at 40th iteration with TT-SVD, TT-ULV (lulv) and TT-URV (lurv), respectively. These algorithms perform well and recover most of the missing data of the original tensor. In addition, Fig.~\ref{fig:RGraderrors} reports that the RSEs of the recovered tensor decline with the iterations, converging to a level of about the prescribed accuracy $0.1$, and that the PSNRs and SSIMs increase during the restoration process. Both the curves and the values in the inset tables are very close, indicating that the RGrad method with TT-UTV is as effective as the Rgrad method with TT-SVD.

\begin{figure}[htbp!]
	\centering
	\subfigure{\includegraphics[width = 0.32\textwidth]{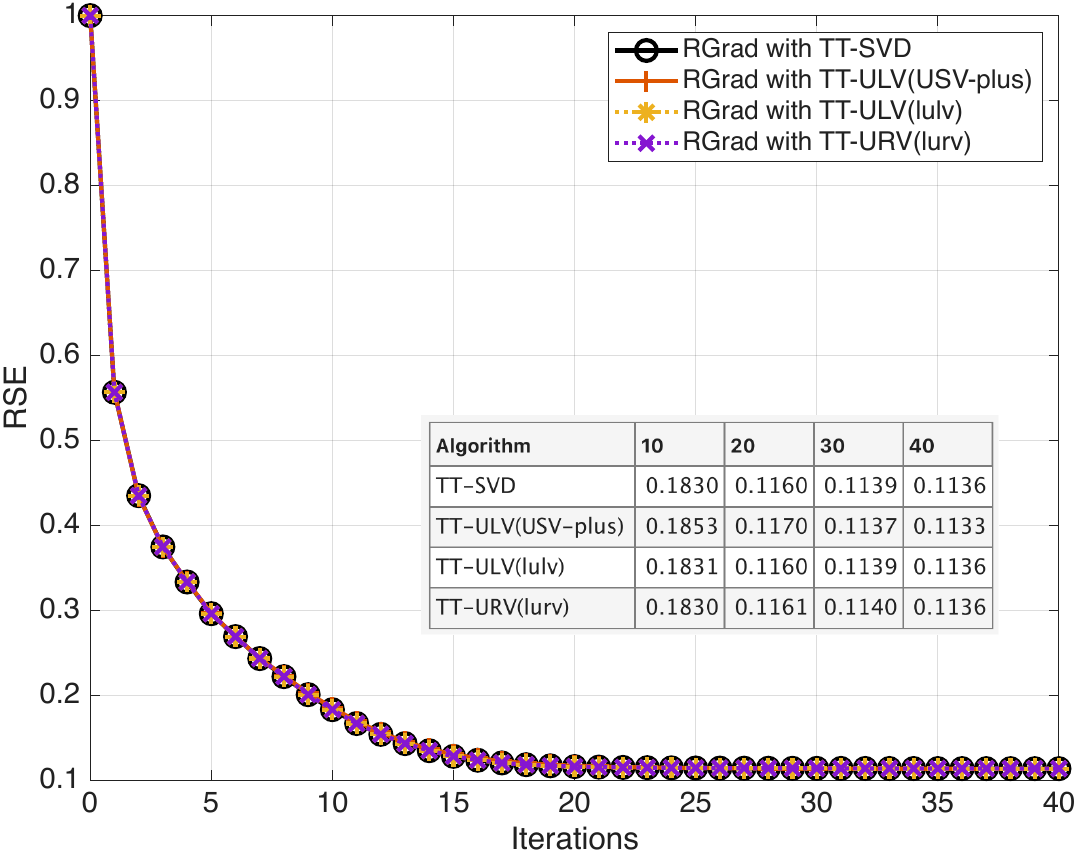}}
	\subfigure{\includegraphics[width = 0.32\textwidth]{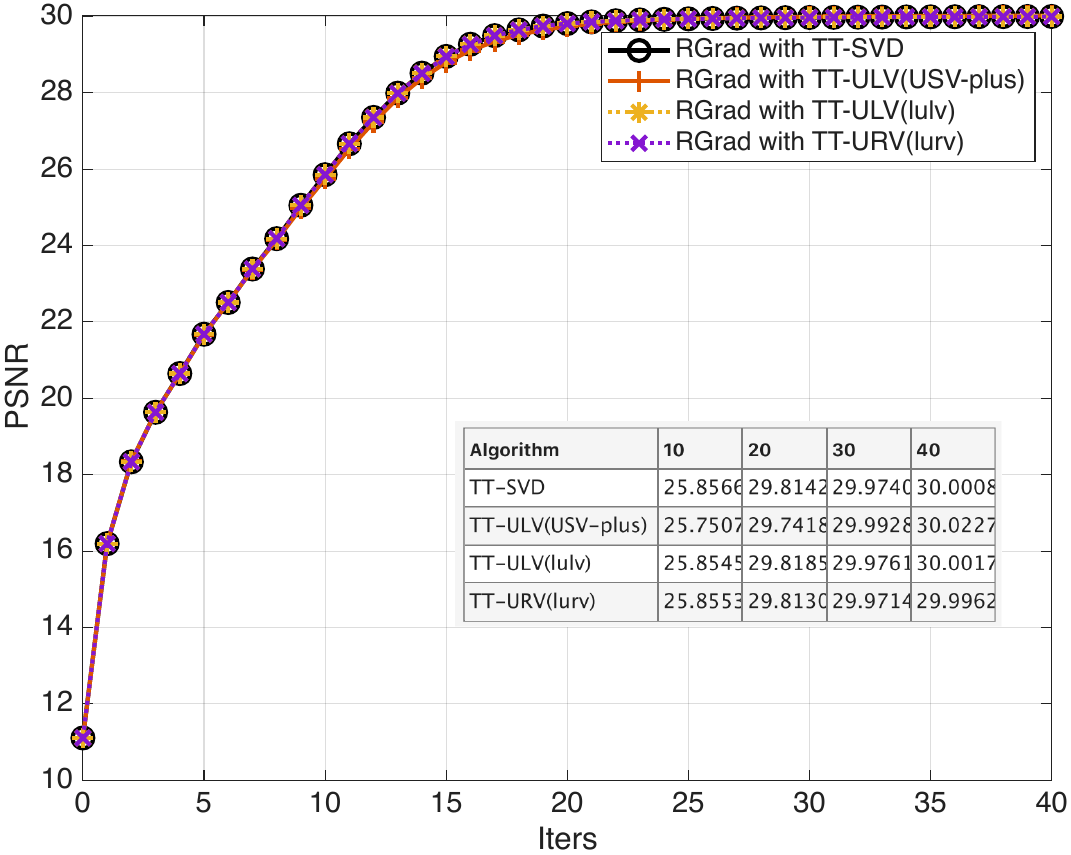}}
    \subfigure{\includegraphics[width = 0.33\textwidth]{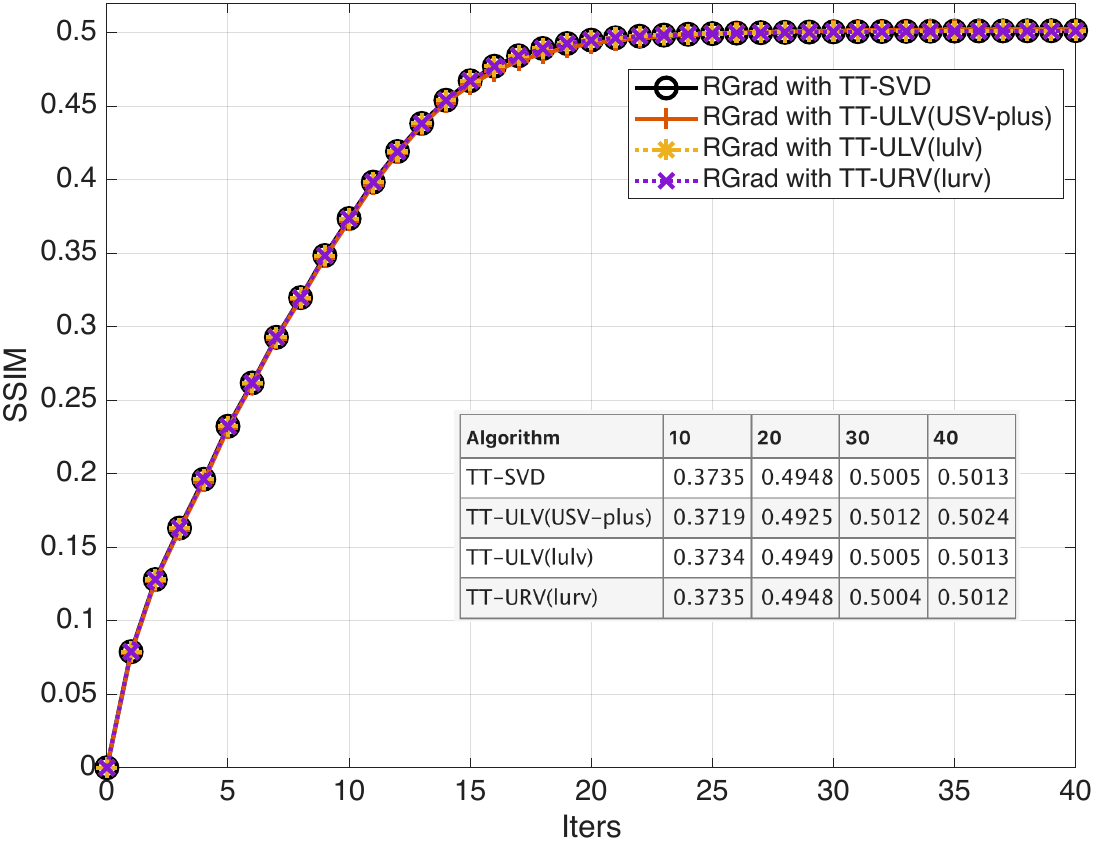}}
	\caption{Evolution of RSEs, PSNRs and SSIMs of the recovered MRI tensor $\widehat{\mathcal{M}}$ over iterations.}
	\label{fig:RGraderrors}
\end{figure}

%\begin{figure}[htbp!]
%	\centering
%	\subfigure[Relative errors]{\includegraphics[width = 0.47\textwidth]{figures/RTTCerr.jpg}}
%	\subfigure[The difference of absolute errors]{\includegraphics[width = 0.47\textwidth]{figures/RTTCabserr.jpg}}
%\end{figure}

%\begin{figure}[htbp!]
%	\subfigure[Original slice]{\includegraphics[width = 0.47\textwidth]{figures/hsv01.jpg}}
%	\subfigure[The slice with missing pixels]{\includegraphics[width = 0.47\textwidth]{figures/hsvMissing01.jpg}} \\
%	\subfigure[RTTC-TTSVD]{\includegraphics[width = 0.47\textwidth]{figures/hsvRTTC01.jpg}}
%	\subfigure[RTTC-TTUTV]{\includegraphics[width = 0.47\textwidth]{figures/hsvRTTC01.jpg}}
%\caption{One slice of the video.}
%\end{figure}
\section{Conclusions}
\label{sec: 5}
%In summary, this article contributes a novel approach called TT-UTV for tensor-train decomposition. These algorithms harness the virtues of UTV decomposition to compute the TT-format tensor, hence require less computational cost over TT-SVD. Through rigorous error analysis, we report the effectiveness of the proposed algorithms, and point out the difference between the ULV and URV cases. We recommend the TT-ULV algorithms in a left-to-right sweep to compute the TT-format with left-orthogonal cores, and the TT-URV algorithm in a right-to-left sweep to compute the TT-format with right-orthogonal cores. We highlight the efficacy and versatility of the proposed algorithms on image compression and MRI data completion, which can provide the comparable accuracy to the TT-SVD.

For computing the TT decomposition, we introduced an algorithmic framework based on the rank-revealing UTV decomposition and the merit of TT-SVD, which yields the proposed TT-UTV algorithms. This framework unifies a broad class of UTV-type decompositions for TT computation, fully leveraging the practical advantages of UTV decompositions over the full SVD in terms of efficiency and flexibility. Theoretically, we establish a tight error bound for the proposed method, recommend distinct sweep strategies for the ULV and URV variants, and derive the corresponding fixed-precision algorithms that adaptively determine the TT ranks. %Numerical experiments show that the proposed TT-UTV algorithms can provide a good trade-off between computational accuracy and complexity, {\color{blue}with accuracy close to that of TT-SVD}; in the fixed-precision setting, the numerical results further validate their reliability and demonstrate the effectiveness of the established theoretical error bounds. {\color{blue}These results are also limited by the available test data; we hope experts will further apply the proposed method in broader applications.}
Numerical experiments show that the TT-UTV algorithms  provide a good trade-off between computational accuracy and complexity, with the accuracy close to that of TT-SVD, and that in the fixed-precision setting they are reliable and consistent with the established error bounds. Admittedly, these tests are also limited by the available data, we hope that broader applications will adopt these algorithms.

UTV-based algorithms are highly adaptable to streaming matrix data, allowing the decomposition of new data to be efficiently updated from previous decomposition at a quadratic computational cost \cite{Stewart1992an}. Future work will focus on integrating these updating strategies into the TT-UTV framework to enable efficient TT decomposition updates for streaming tensor data. Additionally, developing randomized variants for the TT-UTV algorithms represents a promising direction to enhance scalability and real-time processing capability. Motivated by the numerical experiments, updating and optimizing the practical implementation of low-rank UTV algorithms \cite{fierro1999utv} based on contemporary computing platforms is also a worthwhile direction for future research.

\appendix
\section{} \label{sec: appendix}
The TT-ULV algorithm in a right-to-left sweep is given in Algorithm~\ref{alg:TT-ULV2}. Regarding the size of the matrix $\widehat{\mathbf{L}}$ in step 4, we consider two distinct cases. If the low-rank approximation $\widehat{\mathbf{U}}\widehat{\mathbf{L}}\widehat{\mathbf{V}}$ in step 4 is of size $\widehat{\mathbf{L}}\in \mathbb{R}^{I_1\cdots I_{k-1}\times r_{k-1}}$ instead of $\widehat{\mathbf{L}}\in \mathbb{R}^{r_{k-1}\times r_{k-1}}$, that is, the two terms $(\mathbf{U}_{1}\mathbf{L}_{11}+\mathbf{U}_2\mathbf{L}_{21})\mathbf{V}^{\top}_{1}$ in the ULV decomposition \eqref{eq: ulv} are retained to make the matrix $\widehat{\mathbf{V}}=\mathbf{V}_1$ orthogonal to the residual term $\mathbf{E} = \mathbf{U}_2\mathbf{L}_{22}\mathbf{V}^{\top}_2$, then the results similar to Theorem \ref{thm: 3.1} and Corollary \ref{cor: 3.1} also hold for the TT-ULV algorithm~\ref{alg:TT-ULV2} in the right-to-left sweep. But this approach leads to extra matrix multiplication $\mathbf{U}_2\mathbf{L}_{21}$ in each iteration. 

\begin{algorithm}[!htpb]
	\renewcommand{\algorithmicrequire}{\textbf{Input:}}
	\renewcommand{\algorithmicensure}{\textbf{Output:}}
	\caption{TT-ULV algorithm for fixed TT-ranks (right-to-left sweep)}
	\label{alg:TT-ULV2}
	\centering
	\begin{algorithmic}[1]
		\REQUIRE A $d$th-order tensor $\mathcal{A} \in \mathbb{R}^{I_1 \times \cdots \times I_d}$, and fixed TT-ranks $\{ r_0,r_1, \dots, r_d \}$ with $r_0=r_d=1$. 
		\ENSURE Right-orthogonal TT-cores $\mathcal{G}^{(1)}, \dots, \mathcal{G}^{(d)}$ of the approximation $\widehat{\mathcal{A}}$ with TT-ranks $r_k$'s.
		
		\STATE Temporary matrix: $\mathbf{C}=\mathrm{reshape}(\mathcal{A}, \left[ I_1\cdots I_{d-1}, I_d \right])$, i.e., the $(d-1)$th unfolding matrix $\mathbf{A}_{d-1}$.
		\STATE for $k=d$ to $2$ do
		\STATE \quad  $\mathbf{C}:=\operatorname{reshape}\left(\mathbf{C},\left[I_1\cdots I_{k-1}, I_{k}r_{k}\right]\right)$.
		\STATE \quad Compute the truncated rank-$r_{k-1}$ approximation from ULV decomposition $\mathbf{C}=\widehat{\mathbf{U}} \widehat{\mathbf{L}} \widehat{\mathbf{V}}^{\top}+\mathbf{E}$, where $\mathbf{E}$ denotes the residual part in the ULV decomposition \eqref{eq: ulv}.
		\STATE \quad New core: $\mathcal{G}^{(k)}=\operatorname{reshape}\left(\mathbf{V}_1^{\top},\left[r_{k-1}, I_k, r_k\right]\right)$.
		\STATE \quad  $\mathbf{C}= \widehat{\mathbf{U}}\widehat{ \mathbf{L}}$.
		\STATE end for
		\STATE $\mathcal{G}^{(1)}=\mathbf{C}$.
		%\STATE Return tensor $\widehat{\mathcal{A}}$ in TT-format with cores $\mathcal{G}_1, \dots, \mathcal{G}_d$.
	\end{algorithmic}  
\end{algorithm} 

The theory of rank-revealing UTV decomposition tells us that the term $\mathbf{U}_2\mathbf{L}_{21}$ in Eq. \eqref{eq: ulv} can be very minor. Hence, as presented in Algorithm~\ref{alg:TT-ULV2}, the middle matrix $\widehat{\mathbf{L}}$ of the ULV truncation in step 4 should only retain the $\mathbf{L}_{11} \in \mathbb{R}^{r_{k-1}\times r_{k-1}}$ to reduce the computational cost. In contrast, the sharp bound in Corollary \ref{cor: 3.1} cannot be obtained theoretically, since the matrix $\widehat{\mathbf{V}} = \mathbf{V}_1$ used to yield the core tensor is no longer orthogonal to the residual part $\mathbf{E} = \mathbf{U}_2\mathbf{L}_{21}\mathbf{V}_1^{\top} + \mathbf{U}_2\mathbf{L}_{22}\mathbf{V}_2^{\top}$. In this case, a larger bound can be given by a similar analysis of Theorem \ref{thm: 3.1} and the triangle inequality.

\begin{proposition}
	For the TT-ULV algorithm~\ref{alg:TT-ULV2} in the right-to-left sweep, where $\widehat{\mathbf{L}} = \mathbf{L}_{11} \in \mathbb{R}^{r_{k-1}\times r_{k-1}}$ in step 4, if the error of the $k$th truncated ULV decomposition in step 4 is $\varepsilon_k$, then the total error of the computed TT-approximation tensor $\widehat{\mathcal{A}}$ satisfies
	\[ \| \mathcal{A}-\widehat{\mathcal{A}} \|_F \le \sum_{k=2}^{d} \varepsilon_k.\]
\end{proposition}

Based on above results, the TT-ULV algorithm in a right-to-left sweep can also be designed for a prescribed accuracy $\varepsilon$. If the ULV truncation in step 4 is taken as $\mathbf{U}_{1}\mathbf{L}_{11}\mathbf{V}^{\top}_{1}$, then the equal truncated error of each ULV decomposition should be $\varepsilon_k = \frac{\varepsilon}{d-1}\| \mathcal{A} \|_F> \frac{\varepsilon}{\sqrt{d-1}}\| \mathcal{A} \|_F$, which may generate the TT-cores with much higher ranks than TT-ULV algorithm~\ref{alg:TT-UTVver2} in the left-to-right sweep. On the other hand, it can also set $\varepsilon_k = \frac{\varepsilon}{\sqrt{d-1}}\| \mathcal{A} \|_F$, but the middle matrix of ULV truncation
$\widehat{\mathbf{L}}\in \mathbb{R}^{I_1\cdots I_{k-1}\times r_{k-1}}$ should retain a larger size, which leads to more computations and storage. Wherefore, we do not strongly recommend the TT-ULV algorithm for computing right-orthogonal cores using a right-to-left sweep. On the contrary, the TT-URV case is well-suited for a right-to-left sweep; see Section \ref{sec: 3.2} for details.

\section*{Acknowledgment.}  The authors would like to thank the handling editor and
referees for their detailed comments. This work was supported by the National Natural Science Foundation of China (Grant No. U24A2001, 12271108, 12561095, 12601721), the Natural Science Foundation of Sichuan Province (Grant No. 2026NSFSC0752), the Special Posts of Guizhou University (No. [2025]06), and the Guizhou Provincial Major Project of Basic Research Program (Qiankehe Foundation VZD[2026]001).

\bibliographystyle{ieeetr}

%\begin{thebibliography}{00}

%% For numbered reference style
%% \bibitem{label}
%% Text of bibliographic item
\bibliography{ref1}

@article{wang2025tensor,
	title={Tensor-based model reduction and identification for generalized memory polynomial},
	author={Wang, Yuchao and Wei, Yimin},
	journal={IEEE Transactions on Automation Science and Engineering},
	volume={22},
	pages={12655--12667},
	year={2025},
	publisher={IEEE}
}

@article{lestandi2021numerical,
  title={Numerical study of low rank approximation methods for mechanics data and its analysis},
  author={Lestandi, Lucas},
  journal={Journal of Scientific Computing},
  volume={87},
  number={1},
  pages={14},
  year={2021},
  publisher={Springer}
}

@article{sun2020low,
  title={Low-rank Tucker approximation of a tensor from streaming data},
  author={Sun, Yiming and Guo, Yang and Luo, Charlene and Tropp, Joel and Udell, Madeleine},
  journal={SIAM Journal on Mathematics of Data Science},
  volume={2},
  number={4},
  pages={1123--1150},
  year={2020},
  publisher={SIAM}
}

@article{wang2004image,
  title={Image quality assessment: from error visibility to structural similarity},
  author={Wang, Zhou and Bovik, Alan C and Sheikh, Hamid R and Simoncelli, Eero P},
  journal={IEEE Transactions on Image Processing},
  volume={13},
  number={4},
  pages={600--612},
  year={2004},
  publisher={IEEE}
}

@inproceedings{huber2018randomized,
  title={A randomized tensor train singular value decomposition},
  author={Huber, Benjamin and Schneider, Reinhold and Wolf, Sebastian},
  booktitle={Compressed Sensing and its Applications: Second International MATHEON Conference 2015},
  pages={261--290},
  year={2018},
  organization={Springer}
}

@article{che2026efficient,
  title={Efficient Randomized Algorithms for Computing an Approximation of the Tensor Train Decomposition},
  author={Che, Maolin and Wei, Yimin and Yan, Hong},
  journal={Journal of Scientific Computing},
  volume={107},
  number={1},
  note={article no. 2},
  year={2026},
  publisher={Springer}
}

@article{xie2025coupled,
  title={Coupled CPD-Aided Tensor Train Decomposition for {2D-DOD} and {2D-DOA} Estimation in Bistatic {MIMO} Radar},
  author={Xie, Q. and Wen, F. and Wang, X. and Wang, Z. and Yuen, C.},
  journal={IEEE Transactions on Vehicular Technology},
  year={2025},
  Volume = {75},
  Number = {1},
  Pages = {938-952},
  publisher={IEEE}
}

@article{xie2025higher,
  title={Higher-order tensor decomposition for {2D-DOD} and {2D-DOA} estimation in bistatic {MIMO} radar},
  author={Xie, Q. and Shi, J. and Wen, F. and Zheng, Z.},
  journal={Signal Processing},
  Volume = {238},
  note = {article no. 110196},
  year={2025},
  publisher={Elsevier}
}

@article{Che2024tensor,
	Author = {Che, H.  and Pan, B. and Leung, M.-F. and Cao, Y. and Yan, Z.},
	Title = {Tensor Factorization With Sparse and Graph Regularization for Fake News Detection on Social Networks},
	Journal = {IEEE Transactions on Computational Social Sustems},
	Year = {2024},
	Volume = {11},
	Number = {4},
	Pages = {4888-4898},
}

@article{CHEN2022108337,
	title = {Kernelized support tensor train machines},
	journal = {Pattern Recognition},
	volume = {122},
	note = {article no. 108337},
	year = {2022},
	author = {Chen, C. and Batselier, K. and Yu, W. and Wong, N.},
}

@article{Salman2023fast,
	Author = {Ahmadi-Asl, S. and Asante-Mensah, M. G. and Cichocki,
	A. and Phan, A. H. and Oseledets, I. and Wang, J.},
	Title = {Fast cross tensor approximation for image and video completion},
	Journal = {Signal Processing},
	Year = {2023},
	Volume = {213},
	DOI = {10.1016/j.sigpro.2023.109121},
	note = {article no. 109121},
}

@article{YANG2025111580,
	title = {Randomized quaternion tensor {UTV} decompositions for color image and color video processing},
	journal = {Pattern Recognition},
	volume = {165},
	note = {article no. 111580},
	year = {2025},
	issn = {0031-3203},
	doi = {https://doi.org/10.1016/j.patcog.2025.111580},
	author = {Yang, L. and Miao, J. and Jiang, T.-X. and Zhang, Y. and  Kou, K.},
}

@article{Kaloorazi2018Com,
	Author = {Kaloorazi, M. F. and de Lamare, R. C.},
	Title = {{Compressed randomized UTV decompositions for low-rank matrix approximations}},
	Journal = {IEEE Journal of Selected Topics in Signal Processing},
	Year = {2018},
	Volume = {12},
	Number = {6},
	Pages = {1155-1169},
}

@article{xu2023tensor,
  title={Tensor train factorization under noisy and incomplete data with automatic rank estimation},
  author={Xu, L. and Cheng, L. and Wong, N. and Wu, Y.-C.},
  journal={Pattern Recognition},
  volume={141},
  note = {article no. 109650},
  year={2023},
  publisher={Elsevier}
}

@article{ding2018tensor,
  title={{Tensor-based linear dynamical systems for action recognition from 3D skeletons}},
  author={Ding, W. and Liu, K. and Belyaev, E. and Cheng, F.},
  journal={Pattern Recognition},
  volume={77},
  pages={75--86},
  year={2018},
  publisher={Elsevier}
}

@article{lu2011survey,
  title={A survey of multilinear subspace learning for tensor data},
  author={Lu, H. and Plataniotis, K. N. and Venetsanopoulos, A. N.},
  journal={Pattern Recognition},
  volume={44},
  number={7},
  pages={1540--1551},
  year={2011},
  publisher={Elsevier}
}

@article{shi2023parallel,
  title={Parallel algorithms for computing the tensor-train decomposition},
  author={Shi, T. and Ruth, M. and Townsend, A.},
  journal={SIAM Journal on Scientific Computing},
  volume={45},
  number={3},
  pages={C101--C130},
  year={2023},
  publisher={SIAM}
}

@article{zniyed2020tt,
  title={{A TT-based hierarchical framework for decomposing high-order tensors}},
  author={Zniyed, Y. and Boyer, R. and De Almeida, A. LF. and Favier, G.},
  journal={SIAM Journal on Scientific Computing},
  volume={42},
  number={2},
  pages={A822--A848},
  year={2020},
  publisher={SIAM}
}

@article{miron2020tensor,
  title={Tensor methods for multisensor signal processing},
  author={Miron, S. and Zniyed, Y. and Boyer, R. and Lima Ferrer de Almeida, A. and Favier, G. and Brie, D. and Comon, P.},
  journal={IET signal processing},
  volume={14},
  number={10},
  pages={693--709},
  year={2020},
  publisher={Wiley Online Library}
}

@book{golub2013matrix,
  title={Matrix Computations},
  author={Golub, G. H. and Van Loan, C. F.},
  year={2013},
  publisher={Johns Hopkins University Press, Baltimore, MD}
}

@inproceedings{martin2001database,
  title={A database of human segmented natural images and its application to evaluating segmentation algorithms and measuring ecological statistics},
  author={Martin, D. and Fowlkes, C. and Tal, D. and Malik, J.},
  booktitle={Proceedings eighth IEEE international conference on computer vision. ICCV 2001},
  volume={2},
  pages={416--423},
  year={2001},
  organization={IEEE}
}

@article{ko2020fast,
  title={Fast and accurate tensor completion with total variation regularized tensor trains},
  author={Ko, C.-Y. and Batselier, K. and Daniel, L. and Yu, W. and Wong, N.},
  journal={IEEE Transactions on Image Processing},
  volume={29},
  pages={6918--6931},
  year={2020},
  publisher={IEEE}
}

@article{kaloorazi2020efficient,
  title={Efficient low-rank approximation of matrices based on randomized pivoted decomposition},
  author={Kaloorazi, M. F. and Chen, J.},
  journal={IEEE Transactions on Signal Processing},
  volume={68},
  pages={3575--3589},
  year={2020},
  publisher={IEEE}
}

@article{Lee2018rankrev,
    author = {Lee, T. L. and Li, T. Y. and Zeng, Z.},
    title = {{RankRev: a Matlab package for computing the numerical rank and updating/downdating}},
    journal = {Numerical Algorithms},
    year = {2018},
    volume = {77},
    pages = {559-576}
}

@article{stewart1999qlp,
  title={{The QLP approximation to the singular value decomposition}},
  author={Stewart, G. W.},
  journal={SIAM Journal on Scientific Computing},
  volume={20},
  number={4},
  pages={1336--1348},
  year={1999},
  publisher={SIAM}
}

@article{CHAN198767,
title = {{Rank revealing QR factorizations}},
journal = {Linear Algebra and its Applications},
volume = {88},
pages = {67-82},
year = {1987},
author = {Chan, T. F.},
}

@article{tucker1966some,
  title={Some mathematical notes on three-mode factor analysis},
  author={Tucker, L. R.},
  journal={Psychometrika},
  volume={31},
  number={3},
  pages={279--311},
  year={1966},
  publisher={Springer}
}

@article{kilmer2011factorization,
  title={Factorization strategies for third-order tensors},
  author={Kilmer, M. E. and Martin, C. D.},
  journal={Linear Algebra and its Applications},
  volume={435},
  number={3},
  pages={641--658},
  year={2011},
  publisher={Elsevier}
}

@article{che2022efficient,
  title={{An efficient algorithm for computing the approximate t-URV and its applications}},
  author={Che, M. and Wei, Y.},
  journal={Journal of Scientific Computing},
  volume={92},
  number={3},
  pages={93},
  year={2022},
  publisher={Springer}
}

@article{Vannieuwenhoven2012a,
    author = {Vannieuwenhoven, N. and Vandebril, R. and Meerbergen, K.},
    title = {A new truncation strategy for the higher-order singular value decomposition},
    journal = {SIAM Journal on Scientific Computing},
    year = {2012},
    volume = {34},
    number = {2},
    pages = {A1027-A1052}
}

@article{DeLathauwer2000a,
    author = {De Lathauwer, L. and De Moor, B. and Vandewalle, J.},
    title = {A multilinear singular value decomposition},
    journal = {SIAM journal on Matrix Analysis and Applications},
    year = {2000},
    volume = {21},
    number = {4},
    pages = {1253-1278}
}

@article{Vandecappelle2022from,
    author = {Vandecappelle, M. and De Lathauwer, L.},
    title = {{From multilinear SVD to multilinear UTV decomposition}},
    journal = {Signal Processing},
    year = {2022},
    volume = {198},
    pages = {108575}
}

@article{yang1995projection,
  title={Projection approximation subspace tracking},
  author={Yang, B.},
  journal={IEEE Transactions on Signal processing},
  volume={43},
  number={1},
  pages={95--107},
  year={1995},
  publisher={IEEE}
}

@article{Martinsson2019rand,
    author = {Martinsson, P. G. and Quintana-Orti, G. and Heavner, N.},
    title = {{randUTV: A blocked randomized algorithm for computing a rank-revealing UTV factorization}},
    journal = {ACM Transactions on Mathematical Software },
    year = {2019},
    volume = {45},
    number = {1},
    pages = {1-26}
}

@article{Fierro1995bounding,
    author = {Fierro, R. D. and Bunch, J. R.},
    title = {Bounding the subspaces from rank revealing two-sided orthogonal decompositions },
    journal = {SIAM Journal on Matrix Analysis and Applications},
    year = {1995},
    volume = {16},
    number = {3},
    pages = {743-759}
}

@article{Fierro1997low,
    author = {Fierro, R. D. and Hansen, P. C.},
    title = {{Low-rank revealing UTV decompositions}},
    journal = {Numerical Algorithms},
    year = {1997},
    volume = {15},
    pages = {37-55}
}

@article{Stewart1992an,
    author = {Stewart, G. W.},
    title = {An updating algorithm for subspace tracking},
    journal = {IEEE Transactions on Signal Processing},
    year = {1992},
    volume = {40},
    number = {6},
    pages = {1535-1541}
}

@article{Holtz2012the,
    author = {Holtz, S. and Rohwedder, T. and Schneider, R.},
    title = {The alternating linear scheme for tensor optimization in the tensor train format},
    journal = {SIAM Journal on Scientific Computing },
    year = {2012},
    volume = {34},
    number = {2},
    pages = {A683-A713}
}

@article{Sidiropoulos2017,
    author = {Sidiropoulos, N. D. and De Lathauwer, L. and Fu, X. and Huang, K. and Papalexakis, E. E. and Faloutsos, C.},
    title = {Tensor decomposition for signal processing and machine learning},
    journal = {IEEE Transactions on Signal Processing},
    year = {2017},
    volume = {65},
    number = {13},
    pages = {3551-3582}
}

@article{Dolgov2021,
    author = {Dolgov, S. and Kalise, D. and Kunisch, K. K.},
    title = {{Tensor decomposition methods for high-dimensional Hamilton--Jacobi--Bellman equations}},
    journal = {SIAM Journal on Scientific Computing},
    year = {2021},
    volume = {43},
    number = {3},
    pages = {A1625-A1650}
}

@conference{Richter2021Solving,
    author = {Richter, L. and Sallandt, L. and Nüsken, N.},
    title = {{Solving high-dimensional parabolic PDEs using the tensor train format}},
    booktitle = {International Conference on Machine Learning},
    year = {2021},
    pages = {8998-9009}
}

@article{Steinlechner2016,
    author = {Steinlechner, M.},
    title = {Riemannian optimization for high-dimensional tensor completion},
    journal = {SIAM Journal on Scientific Computing},
    volume = {38},
    number = {5},
    year = {2016},
    pages = {S461-S484}
}

@article{Cai2022provable,
    author = {Cai, J. F. and Li, J. and Xia, D.},
    title = {{Provable tensor-train format tensor completion by Riemannian optimization}},
    journal = {Journal of Machine Learning Research},
    volume = {23},
    number = {123},
    year = {2022},
    pages = {1-77}
}

@article{Schollwock2011wave,
    author = {Schollwöck, U.},
    title = {The density-matrix renormalization group in the age of matrix product states},
    journal = {Annals of Physics},
    year = {2011},
    volume = {326},
    number = {1},
    pages = {96-192}
}

@article{novikov2015tensorizing,
  title = {Tensorizing neural networks},
  author = {Novikov, A. and Podoprikhin, D. and Osokin, A. and Vetrov, D. P.},
  journal = {Advances in Neural Information Processing Systems},
  volume = {28},
  year = {2015},
}

@article{Perez-Garcia2007MPS,
  author = {Perez-Garcia, D. and Verstraete, F. and Wolf, M. M. and Cirac, J. I.},
  title = {Matrix product state representations},
  year = {2007},
  publisher = {Rinton Press, Incorporated},
  address = {Paramus, NJ},
  volume = {7},
  number = {5},
  journal = {Quantum Information and Computation},
  pages = {401--430},
  numpages = {30}
}

@conference{Indyk1998,
    author = {Indyk, P. and Motwani, R.},
    title = {Approximate nearest neighbors: towards removing the curse of dimensionality},
    booktitle = {Proceedings of the Thirtieth Annual ACM Symposium on Theory of Computing},
    year = {1998},
    pages = {604-613}
}

@article {cui2024discrete,
	AUTHOR = {Cui, S. and Zhang, G. and Jard\'on-Kojakhmetov,
	H. and Cao, M.},
	TITLE = {On discrete-time polynomial dynamical systems on hypergraphs},
	JOURNAL = {IEEE Control Systems Letters},
	VOLUME = {8},
	YEAR = {2024},
	PAGES = {1078--1083},
}

@article{hore2016tensor,
  title = {Tensor decomposition for multiple-tissue gene expression experiments},
  author = {Hore, V. and Vinuela, A. and Buil, A. and Knight, J. and McCarthy, M. I. and Small, K. and Marchini, J.},
  journal = {Nature Genetics},
  volume = {48},
  number = {9},
  pages = {1094--1100},
  year = {2016},
  publisher = {Nature Publishing Group US New York}
}

@article{fierro1999utv,
  title = {{UTV tools: Matlab templates for rank-revealing UTV decompositions}},
  author = {Fierro, R. D. and Hansen, P. C. and Hansen, P. S. K.},
  journal = {Numerical Algorithms},
  volume = {20},
  number = {2},
  pages = {165--194},
  year = {1999},
  publisher = {Springer}
}

@article{oseledets2010tt,
  title = {{TT-cross approximation for multidimensional arrays}},
  author = {Oseledets, I. and Tyrtyshnikov, E.},
  journal = {Linear Algebra and its Applications},
  volume = {432},
  number = {1},
  pages = {70--88},
  year = {2010},
  publisher = {Elsevier}
}

@article{oseledets2011tensor,
  title = {Tensor-train decomposition},
  author = {Oseledets, I. V.},
  journal = {SIAM Journal on Scientific Computing},
  volume = {33},
  number = {5},
  pages = {2295--2317},
  year = {2011},
  publisher = {SIAM}
}

@article{kolda2009tensor,
  title = {Tensor decompositions and applications},
  author = {Kolda, T. G. and Bader, B. W.},
  journal = {SIAM Review},
  volume = {51},
  number = {3},
  pages = {455--500},
  year = {2009},
  publisher = {SIAM}
}
%\bibitem{lamport94}
%  Leslie Lamport,
%  \textit{\LaTeX: a document preparation system},
%  Addison Wesley, Massachusetts,
%  2nd edition,
% 1994.

%\end{thebibliography}

\end{document}